\documentclass{amsart}
\usepackage[utf8]{inputenc}

\usepackage{enumerate}
\usepackage{enumitem}
\usepackage{multicol}
\usepackage{amssymb,amsmath}
\usepackage{amsthm}
\usepackage{turnstile}
\usepackage{textcomp}
\usepackage{graphicx}
\usepackage{subcaption}
\usepackage{mathtools}
\usepackage{wrapfig}
\usepackage{hhline}
\usepackage{array,multirow}
\usepackage{xcolor}
\usepackage{float}
\usepackage{geometry}
\usepackage{tikz}

\usepackage{hyperref}
\usepackage[normalem]{ulem}

\usetikzlibrary{positioning,calc}

\theoremstyle{plain}
\newtheorem{theorem}{Theorem}[section]
\newtheorem{lemma}[theorem]{Lemma}
\newtheorem{corollary}[theorem]{Corollary}
\newtheorem{definition}[theorem]{Definition}
\newtheorem{notation}[theorem]{Notation}
\newtheorem{fact}[theorem]{Fact}
\newtheorem{remark}[theorem]{Remark}

\DeclareMathOperator{\prcat}{\mathtt{PRCatSpec}}
\DeclareMathOperator{\cone}{\mathtt{Cone}}
\DeclareMathOperator{\shuffle}{\mathtt{Shuffle}}
\DeclareMathOperator{\interval}{\mathtt{Int}}

\title{Primitive recursive categoricity spectra}

\author{Nikolay Bazhenov}
\address{Novosibirsk State University, Novosibirsk, Russia}
\address{Nazarbayev University, Astana, Kazakhstan}
\email{nickbazh@yandex.ru}

\author{Heer Tern Koh}
\address{Nanyang Technological University, Singapore}
\email{heertern001@e.ntu.edu.sg}

\author{Keng Meng Ng}
\address{Nanyang Technological University, Singapore}
\email{kmng@ntu.edu.sg}

\thanks{Ng was supported by the Ministry of Education, Singapore, under its Academic Research Fund Tier 2 (MOE-T2EP20222-0018) and Academic Research Fund Tier 1 (RG104/24). We also thank David Belanger for the many helpful discussions.}

\date{\today}

\begin{document}

\begin{abstract}
We study the primitive recursive analogue of computable categoricity spectra for various natural classes of structures. We show that these notions coincide for all relatively $\Delta_{2}^{0}$-categorical equivalence structures and linear orders, relatively $\Delta_{3}^{0}$-categorical Boolean algebras, and computably categorical trees as partial orders.
\end{abstract}

\maketitle

\section{Introduction}

A wildly successful application of computability theory has been in the study of discrete mathematical structures. The main objects of interest are \emph{computable structures}, which were first defined in \cite{rab60,mal61} as structures with finitely many relations and functions such that each of the relations and functions are computable. Such structures are also often referred to as having computable presentations. A research program in the study of computable structures is the investigation of \emph{categoricity} (see for example \cite{Gon-80,gd80,remmel81b,remmel81,mccoy03,cchm06,cchm09,chr14,fro15,Baz-21}). Roughly speaking, such investigations provide insight into how `algorithmically complicated' isomorphisms between (classically) isomorphic computable structures are. For instance, if between any two computable presentations of a structure, there exists a computable isomorphism, then we say that such a structure is \emph{computably categorical}. Intuitively, this means that we may obtain an isomorphism between two given copies of the structure in an algorithmic way, and thus, the isomorphisms of such structures are relatively `simple'.

It was observed in \cite{kmn17} that in natural classes of mathematical structures, merely having a computable presentation also guarantees the existence of an `efficient' presentation \cite{grigorieff90,cr91,cr92,cdru09}. Additionally, a key step in extracting an `efficient' presentation from a computable one is in eliminating the use of unbounded search. Thus, to better understand such a phenomenon, it was proposed to investigate mathematical structures from the point of view of \emph{primitive recursive functions}, algorithms that omit the use of unbounded search. Under such a setting, researchers study \emph{punctual structures}, structures with domain $\omega$ such that all of its functions and relations are uniformly primitive recursive. In contrast to computable structures, punctual structures should reveal their points or elements `without delay', reflecting the intuition that a punctual presentation is `faster' or more feasible than a computable one. One of the first theorems regarding punctual structures is as follows.

\begin{theorem}[\cite{kmn17}]\label{thm:puncpres}
    In each of the following classes, every computable structure has a punctual presentation: equivalence structures, linear orders, Boolean algebras, torsion-free Abelian groups, and Abelian $p$-groups.
\end{theorem}

In some sense, this reflects the phenomenon that computable structures in natural classes also possess efficient presentations. Following the pattern in computable structure theory, one may consider the question of categoricity for punctual structures. To investigate such questions, we must first fix an appropriate notion of isomorphism. In computable structure theory, computable structures are studied up to computable isomorphisms. As such, punctual structures should be studied up to isomorphisms computable without delay. However, since the inverse of a primitive recursive function is not necessarily primitive recursive, researchers consider \emph{punctual isomorphisms} instead: primitive recursive isomorphisms with primitive recursive inverse. Analogous to computably categorical structures, a structure is \emph{punctually categorical} if there is a punctual isomorphism between any two punctual presentations of such a structure. Just as before, this means that the isomorphisms on such structures are `algorithmically simple'.

\begin{theorem}[\cite{kmn17}]\label{thm:punccat}
\
\begin{enumerate}
    \item An equivalence structure $S$ is punctually categorical iff it is either of the form $F\cup E$, where $F$ is finite and $E$ has only classes of size $1$, or $S$ has only finitely many classes, at most one of which is infinite.

    \item A linear order is punctually categorical iff it is finite.

    \item A Boolean algebra is punctually categorical iff it is finite.
\end{enumerate}
\end{theorem}

As it turns out, the only structures (in natural classes) which are punctually categorical are those that are finitistic in some sense. This suggests that the algorithmic nature of various back and forth arguments of categoricity (for infinite structures) like the countable dense linear order with no endpoints, or the random graph, cannot be carried out without unbounded search. In this way, studying mathematical structures using primitive recursive functions provides a new and interesting viewpoint. (See the surveys \cite{bdkmn19,dmn21} for some of the research directions regarding punctual structures.)

\subsection{Preliminaries}\label{sec:prelim}

Recall that the \emph{degree of categoricity} of a computable structure is the least Turing degree $\mathbf{d}$ which computes an isomorphism between any two computable presentations of the structure \cite{fkm10,gon11}. Evidently, to show that some structure has degree of categoricity $\mathbf{d}$, one must first show that $\mathbf{d}$ is sufficient to produce isomorphisms between any two computable presentations of the given structure. Second, to show that such a degree is `sharp', one typically constructs two computable presentations so that any isomorphism $f$ between these presentations is such that $f\geq_{T}\mathbf{d}$. Roughly speaking, the degree of categoricity of a computable structure measures how algorithmically complex isomorphisms between computable presentations of the given structure are. In order to study the primitive recursive analogue of such notions, we require the following definition.

\begin{definition}[\cite{km21}]\label{def:pr}
    Let $f,g:\omega\to\omega$ be total functions. $f$ reduces to $g$, written $f\leq_{PR}g$ if there is some primitive recursive scheme $\Psi$ such that $\Psi^{g}=f$. If we further have that $g\leq_{PR}f$, then we write $f\equiv_{PR}g$. The induced degree structure shall be referred to as the PR-degrees.
\end{definition}

Just like how Turing degrees are a measure of non-computability, the PR-degrees analogously measures the primitive recursive content of various functions. Intuitively, if $f\leq_{PR}g$, then $g$ grows faster than $f$ does, and is thus more `primitive recursively complex' than $f$. Under this setting, and following the ideas and terminology from computable structure theory, we define:

\begin{definition}[\cite{bk21}]\label{def:catspec}
    For a given punctual structure $A$, we use $\prcat(A)$ to denote the set containing all PR-degrees $\mathbf{d}$ such that for any punctual structure $B\cong A$, there is an isomorphism $f:A\to B$ and $\mathbf{d}\geq_{PR}f,f^{-1}$.
\end{definition}

As noted earlier, since the inverse of primitive recursive functions is not necessarily primitive recursive, it follows that for a given function $f$, $f$ and $f^{-1}$ may belong to different PR-degrees. As such, when defining the primitive recursive categoricity spectrum of a structure as above, we require that each PR-degree $\mathbf{d}\in\prcat(A)$ primitive recursively computes both an isomorphism and its inverse.

In computability theory, the Turing degrees $\mathbf{0},\mathbf{0}',\mathbf{0}'',\dots$ serves as a `spine' to classify and rank the algorithmic content of various objects or processes. In a similar vein, we define the following collections of PR-degrees:

\begin{definition}[\cite{bk21}]\label{def:cone}
    For each computable ordinal $\alpha>0$, $\cone(\Delta_{\alpha}^{0})$ is the set containing all PR-degrees $\mathbf{d}$ such that $\mathbf{d}\geq_{PR}f$ for any total $\Delta_{\alpha}^{0}$-function $f$.
\end{definition}

From Definitions \ref{def:cone} and \ref{def:catspec}, it is not difficult to obtain the following:

\begin{fact}\label{fact:conesubspec}
    If $A$ is $\Delta_{\alpha}^{0}$-categorical, then $\cone(\Delta_{\alpha}^{0})\subseteq\prcat(A)$.
\end{fact}

In some sense, this is the analogue of $\Delta_{\alpha}^{0}$-categoricity of computable structures for punctual structures. In order to show that this bound is sharp, i.e., that $\Delta_{\alpha}^{0}$ is the least primitive recursive degree of categoricity for $A$, we show that $\prcat(A)\subseteq\cone(\Delta_{\alpha}^{0})$, and hence $\prcat(A)=\cone(\Delta_{\alpha}^{0})$. This means that the PR-degrees which primitive recursively compute isomorphisms of $A$ exactly coincide with the PR-degrees which primitive recursively compute $\Delta_{\alpha}^{0}$-functions, and thus, is analogous to the notion of degree of categoricity for computable structures.

We recall that a computable structure $A$ is \emph{relatively $\Delta^0_{\alpha}$-categorical if for any countable copy $B \cong A$, there exists an isomorpishm $f\colon B\cong A$ such that $f \in \Delta^0_{\alpha}(B)$. Relative $\Delta^0_{\alpha}$-categoricity always implies $\Delta^0_{\alpha}$-categoricity, but not vice versa \cite{Gon-77,GHK-05,CFG-09}.}

\

In the present article, we expand upon the work initiated in \cite{km21,bk21,bk24}, investigating the primitive recursive complexity of isomorphisms between punctual structures. We show that in the classes of equivalence structures, linear orders, and Boolean algebras, being relatively $\Delta_{n}^{0}$-categorical generally coincides with $\prcat(A)=\cone(\Delta_{n}^{0})$. More specifically, we show that the respective notions correspond nicely up to relatively $\Delta_{2}^{0}$-categorical equivalence structures; relatively $\Delta_{2}^{0}$-categorical linear orders; relatively $\Delta_{3}^{0}$-categorical Boolean algebras. We also study some relatively $\Delta_{3}^{0}$-categorical linear orders and computably categorical trees as partial orders. One may be tempted to think that such notions always coincide, however, in the companion paper \cite{bbkn25b}, we provide some counterexamples.

\section{Equivalence Structures}\label{sec:equiv}

In \cite{cchm06}, it was shown that the equivalence structures that are computably categorical coincide with the relatively computably categorical ones, and these are exactly one of the following types of equivalence structures:
\begin{itemize}
    \item $E$ has only finitely many finite classes (and possibly cofinitely many infinite classes);
    \item $E$ has only finitely many infinite classes and at most one finite $k$ such that there are infinitely many classes of size $k$.
\end{itemize}
Applying Fact \ref{fact:conesubspec}, such equivalence structures $E$ evidently have the property that $\cone(\Delta_{1}^{0})\subseteq\prcat(E)$. It is not difficult to see that if $E$ is punctually categorical, then $\prcat(E)\not\subseteq\cone(\Delta_{1}^{0})$, as there exist total computable functions which are not primitive recursive. Thus, we consider only equivalence structures which are not punctually categorical in the following theorem.

\begin{theorem}\label{theo:Delta_1-cat-eq-str}
    If $E$ is a relatively $\Delta_{1}^{0}$-categorical equivalence structure that is not punctually categorical, then $\prcat(E)=\cone(\Delta_{1}^{0})$.
\end{theorem}

\begin{proof}
Recall from Theorem \ref{thm:punccat} that an equivalence structure is punctually categorical iff it has at most one infinite class and finitely many other classes, or it has only finite classes with cofinitely many of size $1$. In order for an equivalence structure, $E$, to be relatively $\Delta_{1}^{0}$-categorical and not punctually categorical, it has to be one of the following types. 
\begin{enumerate}[label=(\roman*)]
    \item $E$ has cofinitely many classes of size $1$. In order for $E$ to not be punctually categorical, $E$ must also have at least one infinite class.

    \item Not the previous type; $E$ has cofinitely many classes of the same size $>1$.

    \item Finitely many classes with at least two infinite classes.
\end{enumerate}
Roughly speaking, for a given total $g$, to define $h$ so that $g\leq_{PR}h$, we encode the halting times of $g(x)$ in $h(x)$. Then, to primitive recursively recover $g(x)$ from $h(x)$, we simply compute $g(x)$ until the $h(x)$-th stage and return the corresponding value. To show that $\prcat(E)\subseteq\cone(\Delta_{1}^{0})$, given any total computable function $g$, we construct equivalence structures $A,B\cong E$ such that $g\leq_{PR}h\oplus h^{-1}$ for any isomorphism $h:A\to B$.

First consider equivalence structures of Type (i). We construct $A,B\cong E$ as follows. At stage $0$, enumerate all classes of finite size $\neq 1$ into both $A$ and $B$, and non-uniformly fix the number of infinite classes in $E$. At each stage $s$, do the following.
\begin{enumerate}
    \item Enumerate a new element into each infinite class in both $A$ and $B$. This ensures that $A$ and $B$ remain punctual.

    \item Enumerate a new class of size $1$ into $A$.

    \item If for some $x,\,g(y)[s]\downarrow$ for all $y<x$, then enumerate the $x$-th class of size $1$ into $B$ (if it has yet to be enumerated).
\end{enumerate}
The intuition here is obviously to keep $A$ `standard' while delaying enumerating classes of size $1$ into $B$, thus encoding $g$. Let $h:A\to B$ be an isomorphism. To recover $g(x)$, compute $h$ on the first $x+1$ many classes in $A$ of size $1$. Observe that the indices of such elements can be recovered in a primitive recursive way by enumerating $A$ up to stage $x+1$. Since $h$ is an isomorphism, this has to produce the indices of $x+1$ many distinct classes of size $1$ in $B$. By construction, the $x+1$-th class of size $1$ in $B$ is enumerated only after $g(y)[s]\downarrow$ for all $y<x+1$ is witnessed. In particular, at least one of the $x+1$ many distinct classes obtained from $h$ has index at least as big as the stage at which $g(x)\downarrow$.

Now consider equivalence structures of Type (ii). Once again, we shall construct $A,B\cong E$. Let $N>1$ be the size (possibly infinite) which is repeated cofinitely often in $E$. As before, we enumerate all classes not of this size $N$ in both $A$ and $B$ in a `standard' way. During stage $s$, do the following.
\begin{enumerate}
    \item If $N$ is finite, then enumerate a class of size $N$ into $A$. Otherwise, enumerate a new class into $A$, and for each class not in the `standard' part, ensure that it contains $s+2$ many elements. In the limit, it is obvious that $A\cong E$.

    \item Enumerate a new class into $B$, leaving it temporarily at size $1$. This ensures that $B$ remains punctual.

    \item If there is some $x$ such that for all $y<x,\,g(y)[s]\downarrow$, then ensure that the first $x$ many classes of $B$ contains $N$ many elements, if $N$ is finite, or $s+2$ many elements otherwise.
\end{enumerate}
The argument here is similar to the one before. $A$ is the `fast' copy, producing elements on which we encode $g$ in a primitive recursive way. $B$ on the other hand is the `delayed' copy, only enumerating elements when $g\downarrow$. Instead of `delaying' via growing only an infinite class in $B$ like in the strategy for Type (i) equivalence structures, we `delay' $B$ via enumerating new classes and leaving them at size $1$. Once $g(x)\downarrow$, then we grow the classes to the size $N$.

Let $h:A\to B$ be an isomorphism. To compute $g(x)$, first compute $h$ on the first two elements of the first $x+1$ many classes of size $N$. This obviously produces the indices of two elements in $x+1$ many distinct classes of size $N$ in $B$. Since the element with the second least index in the $x+1$-th class in $B$ is at least the stage at which $g(y)\downarrow$ for all $y<x+1$, by taking the maximum over the indices of elements obtained by computing $h$ allows us to recover $g(x)$.

Finally, we consider equivalence structures of Type (iii). As there are only finitely many finite classes, we may non-uniformly fix them and enumerate them into both $A$ and $B$ at stage $0$. Let $M\geq 2$ be the number of infinite classes. At stage $s$, do the following.
\begin{enumerate}
    \item In $A$, enumerate $M$ new elements, adding one to each infinite class of $A$.

    \item In $B$, only enumerate $1$ new element, adding it to the first infinite class of $B$. For the rest of the $M-1$ infinite classes in $B$, we keep them at the size $x$ where at stage $s$, $g(y)[s]\downarrow$ for all $y<x$.
\end{enumerate}
A similar argument as before works. Let $h:A\to B$ be an isomorphism. To find $g(x)$, simply compute $h$ on the first $x+1$ many elements of two infinite classes in $A$. By pigeonhole principle, at least one such class must be mapped by $h$ to a `slow' infinite class in $B$. Furthermore, one such image has index at least as large as the stage at which $g(x)\downarrow$. Theorem~\ref{theo:Delta_1-cat-eq-str} is proved.
\end{proof}

\

\begin{theorem}\label{thm:d2equiv}
    If $E$ is a relatively $\Delta_{2}^{0}$-categorical equivalence structure which is not relatively $\Delta^0_1$-categorical, then $\prcat(E)=\cone(\Delta_{2}^{0})$.
\end{theorem}

Applying the characterisation in \cite[Corollary 4.8]{cchm06}, the relatively $\Delta_{2}^{0}$-categorical equivalence structures have either bounded character (this means that the equivalence structure has only finitely many different finite class sizes), or finitely many infinite classes. Evidently, we may assume that the equivalence structure has infinitely many classes, otherwise it would be computably categorical. We split the proof of Theorem \ref{thm:d2equiv} into Lemmas \ref{lem:d2boundchar} and \ref{lem:d2fininf}, addressing the two possibilities respectively.

\begin{lemma}\label{lem:d2boundchar}
    If $E$ is an equivalence structure with bounded character and infinitely many classes, then $\prcat(E)=\cone(\Delta_{2}^{0})$.
\end{lemma}

\begin{proof}
If $E$ has cofinitely many infinite classes, then $E$ is computably categorical. We may thus assume that $E$ has infinitely many finite classes. Since $E$ has bounded character, let $M\in\omega$ be the least size such that $E$ has infinitely many classes of size $M$. For the same reasons as before, $M$ cannot occur cofinitely often. For a given total $\Delta_{2}^{0}$-computable function, $g^{*}$, we consider the primitive recursive approximation $g(x,s)$ so that $\lim_{s}g(x,s)=g^{*}(x)$. Once again, in order to recover the value of $g^{*}(x)$ primitive recursively, we encode the \emph{stabilising stage} $s_{x}$, such that for each $s\geq s_{x},\,g(x,s)=g(x,s_{x})$. As long as we have a way to obtain $s_{x}$, we may primitive recursively recover the value of $g^{*}(x)=\lim_{s}g(x,s)=g(x,s_{x})$.

\

\emph{Encoding $s_{x}$:} Given a computable equivalence structure $E$ with bounded character and infinitely many classes, we construct $A,B\cong E$ punctual equivalence structures with the following properties. $A$ will enumerate classes of size $M$ `quickly'. Every even indexed class of $A$ will be a class of size $M$, while we `copy' $E$ into the odd indexed classes of $A$. Since $E$ consists of infinitely many classes of size $M$, the resulting structure is clearly still isomorphic to $E$.

In contrast, $B$ will produce classes of size $M$ extremely sparsely. For each $x$, whenever $g(x,s)\neq g(x,s-1)$, we ensure that within the first $s$ many classes of $B$, there are at most $x$ many classes of size $M$. By the choice of $M$, we may in fact also assume that it is the smallest class size in $E$, as there are only possibly finitely many classes of smaller size, which can be fixed non-uniformly. That is, we are able to decrease the number of classes of size $M$ in the first $s$ many classes by increasing the classes of size $M$ to some larger size. By computing the given isomorphism on the first $x+1$ many classes of size $M$ in $A$, it is not too hard to see that the stabilising stage $s_{x}$ may be recovered. The only remaining concern is how this strategy might be executed while ensuring $B\cong E$.

\

\emph{Ensuring $B\cong E$:} When constructing $B$, other than the possibly finitely many classes of smaller size, we alternately enumerate classes of size $M$ and $M+1$ into $B$. We shall build a computable (not total) function $f:B\to E$ such that the domain and range of $f$ are exactly those classes of $B$ and $E$ respectively that are not of size $M$. Let $e_{i}$ denote the $i^{th}$ class enumerated into $E$. Whenever we need to grow some class $b$ in $B$ to a size $>M$ (as dictated by the strategy for encoding $s_{x}$), first grow the class to size $M+1$. Then, search for the least $i$ where $e_{i}$ currently has size $>M$ and is not in the range of $f$. Once found, define such an $e_{i}$ to be the $f$ image of $b$. For all subsequent stages, let the size of $b$ be the same as the size of $e_{i}$. With some care, it can be easily arranged to define $f$ on all classes in $B$ which grows to a size $>M$. Similarly, all classes in $E$ of size $>M$ must eventually enter the range of $f$, provided that there are infinitely many classes in $B$ of size $>M$. That is, $f$ is an isomorphism between the classes of $B$ and $E$ of size $>M$. It follows from the assumption that there exists infinitely many such classes, otherwise $E$ will be computably categorical.

\

\emph{Construction:} At stage $0$, enumerate the finitely many classes of size $<M$ into both $A$ and $B$. All subsequent classes in $A$ and $B$ shall be denoted as $a_{0},a_{1},\dots$ and $b_{0},b_{1},\dots$ respectively. At stage $s$, enumerate the classes $a_{s}$ and $b_{s}$ into $A$ and $B$ respectively. If $s$ is even, then let the size of both $a_{s}$ and $b_{s}$ be $M$. Otherwise, let the size of $a_{s}$ be the size of $e_{(s-1)/2}$, and let the size of $b_{s}$ be $M+1$. Let $x<s$ be the least for which $g(x,s)\neq g(x,s-1)$. If no such $x$ exists, then proceed to stage $s+1$. Otherwise, let $b_{i}$ denote the $x^{th}$ class of size $M$ in $B$, and we increase the size of $b_{j}$ for each $i\leq j<s$ of size $M$ to $M+1$.

At each stage $s$, also pick the class $b$ in $B$ of size $M+1$ with the least index such that $f(b)\uparrow$. Search for the least $i$ for which $e_{i}$ currently has size $>M$ and is not yet in the range of $f$. Once found, define $f(b)=e_{i}$, and for all subsequent stages, let the size of $b$ be the same as the size of $e_{i}$. If no such $e_{i}$ can be found, then proceed to stage $s+1$ (note that this search is primitive recursive).

\

\emph{Verification:} It follows directly from the construction that $A\cong E$. Now we claim that $f$ is an isomorphism between all classes in $B$ and $E$ of size $>M$. $f$ is clearly injective, and once $f(b)\downarrow$, $b$ is maintained to be the same size as $f(b)$. For any given $b$ of size $>M$, it is clear that it must eventually be the class in $B$ of lowest index; the construction will attempt to define $f(b)$ at the next possible instance. Since $E$ is has infinitely many classes of size $>M$, an appropriate $f$ image for $b$ must always be found. In addition, as we always pick the class with the least possible index as the image, $f$ is also surjective on classes in $E$ of size $>M$. Thus, $f$ is an isomorphism between all classes of size $>M$ in $B$ and $E$.

Now we prove by induction that there exists infinitely many classes of size $M$ in $B$, and that the $x^{th}$ such class exists and has index $\geq s_{x}$, the stabilising stage for $x$. Since $g^{*}$ is assumed to be total, $s_{0},s_{1},\dots$ all exist. We further assume that $s_{x}$ is the least possible stage at which $g(x,s)$ stabilises.

For any $i<s_{0}$, at stage $s_{0}$, we must have discovered that $g(0,s_{0})\neq g(0,s_{0}-1)$, and thus changed the size of $b_{i}$ still of size $M$ to $M+1$. For all subsequent stages, observe that the construction never changes the size of $b_{2k}$ where $k$ is the least for which $2k\geq s_{0}$. Any change in $g(x,s)$ for $x>0$ only affects the $y^{th}$ class of size $M$ for $y\geq x$.

Now assume inductively that for each $y<x$, the $y^{th}$ class of size $M$ exists, and has index $\geq s_{y}$. Let these classes be indexed by $b_{i_{0}},b_{i_{1}},\dots,b_{i_{x-1}}$. For stages $s>s_{x-1}$, we may possibly only discover $g(z,s)\neq g(z,s-1)$ for some $z\geq x$. By the construction, whenever such $z$ and $s$ are found, we only change the classes with index at least as big as the current index of the $z^{th}$ class of size $M$. In other words, no such action may change the sizes of the classes $b_{i_{y}}$ for any $y<x$. By the construction, at stage $s_{x}$, observe that any class of size $M$ with index $i$ where $i_{x-1}<i<s_{x}$ will be changed to size $M+1$, as we discovered that $g(x,s_{x})\neq g(x,s_{x}-1)$. After such a stage $s_{x}$, the size of $b_{2k}$ where $k$ is the least such that $2k\geq s_{x}$ remains forever at $M$.

Applying the claim allows us to conclude that $B$ has infinitely many classes of size $M$, and is thus isomorphic to $E$, and also that for each $x$, the $x^{th}$ class of size $M$ in $B$ has index $\geq s_{x}$. With the latter, we may recover $g^{*}$ as follows. Let $h:A\to B$ be an isomorphism. Given $x$, compute $h(a_{i})$ for each $i<2x$. Recall that $a_{i}$ has size $M$ for even $i$ and are classes from the given computable presentation $E$ for odd $i$. That is, there are at least $x+1$ many classes of size $M$ on which we compute $h$. By pigeonhole principle, at least one of these $h$ images must land on a class of $B$ with index $\geq s_{x}$. By computing $g(x,s)$ where $s$ is the maximum of the indices of all $h(a_{i})$ for each $i<2x$, we obtain $g^{*}(x)$. Lemma~\ref{lem:d2boundchar} is proved.
\end{proof}

\

\begin{lemma}\label{lem:d2fininf}
    If $E$ is an equivalence structure with finitely many infinite classes and unbounded character, then $\prcat(E)=\cone(\Delta_{2}^{0})$.
\end{lemma}

\begin{proof}
Let $E$ be a computable equivalence structure with finitely many infinite classes and unbounded character. By non-uniformly fixing the number of infinite classes in $E$ and the indices of some representatives, we may assume that $E$ only has classes of finite size. As usual, we construct punctual equivalence structures $A,B\cong E$ such that a given total $\Delta_{2}^{0}$ function may be computed primitive recursively from any isomorphism $h:A\to B$. Let $g^{*}$ be a total $\Delta_{2}^{0}$ function and let $g(x,s)$ be a primitive recursive approximation to $g^{*}$. For each $x$, the \emph{stabilising} stage is the stage $s_{x}$ where $g(x,s)=g(x,s_{x})$ for all $s\geq s_{x}$. The goal is to encode for each $x$, the stabilising stage $s_{x}$ into the structures $A,B$.

\

\emph{Encoding $s_{0}$:} We denote the equivalence classes of $A,B$ in the order they are enumerated as $a_{i}$ and $b_{i}$ respectively. Similarly denote the equivalence classes of the computable copy $E$ as $e_{0},e_{1},\dots$. During the construction, $A$ will be copying $E$; at each stage $s\geq i$, $a_{i}$ will have size either $1$ or the current size of $e_{i}$, whichever is larger. Observe that even if $E$ does not grow any of its existing classes nor enumerates new classes, $A$ remains punctual.

Given any isomorphism $h:A\to B$, we want $h(a_{0})$ to have an index at least as large as $s_{0}$. We call $a_{0}$ the \emph{witness} for $R_{0}$ (we shall introduce the requirements shortly). To ensure that this property holds, whenever $g(0,s)\neq g(0,s-1)$, we \emph{tag} each class $b_{i}$ for all $i<s$ with $0$. This means that each of these classes should be kept at a size strictly larger than $a_{0}$. The idea here is that since $E$ has unbounded character, we will eventually be able to find such classes of larger size in $E$ and keep $B\cong E$. To this end, we define a $\Delta_{2}^{0}$ map $p:B\to A$ to aid us in constructing $B$. For each class $b_{i}$ tagged with $0$, define $p(b_{i})=a_{j}$ where $j$ is the least for which $a_{j}$ is not yet in the range of $p$ and also has size currently larger than that of $a_{0}$ and at least as large as that of $b_{i}$. For each subsequent stage, we keep the size of $b_{i}$ the same as the size of $p(b_{i})$. Similarly, define $p(b_{s})=a_{0}$ and keep $b_{s}$ the same size as $a_{0}$. Evidently, $b_{s}$ will be the first class in $B$ to have the same size as $a_{0}$; any isomorphism $h:A\to B$ must map $a_{0}$ to a class with index at least as large as $b_{s}$. For bookkeeping purposes, we will also tag such classes with a special character $0^{\dagger}$, or more generally, $x^{\dagger}$.

Since $g^{*}$ is assumed to be total, the number of classes tagged with $0$ must be finite, as $s_{0}$ is exactly the number of classes in $B$ tagged with $0$. For each such class $b_{i}$ tagged with $0$, since we want to keep their sizes strictly larger than that of $a_{0}$, whenever the size of $a_{0}$ increases, we might need to redefine $p(b_{i})$. However, since $E$ is assumed to have only finite classes, the size of $e_{0}$ must eventually stop increasing. That is, the size of $a_{0}$ also stabilises at some finite stage. After such a stage, no further changes will be made to $p(b_{i})$ for each $b_{i}$ tagged with $0$.

\

\emph{Encoding $s_{x}$:} As the construction is considerably more complex here, we employ a simple priority argument to meet the following requirements.
\begin{align*}
    P:&\,B\cong E\\
    R_{x}:&\,\text{There exists }a\in A,\text{ such that the index of }h(a)\text{ is at least as big as }s_{x},\\
    &\,\text{for any isomorphism }h:A\to B.
\end{align*}
In addition, we call the $a\in A$ which meets $R_{x}$ the \emph{witness} for $R_{x}$. Suppose that for each $y<x$, we have found witnesses $a_{i_{y}}$ for $R_{y}$. Recall that the strategy was to keep the indices of the possible images for the various $a_{i_{y}}$ at least as large as $s_{y}$, and that this is accomplished by growing the classes tagged with $y$ to be bigger than the size of $a_{i_{y}}$. This means that we must take some care in the choice of witness for $R_{x}$. If we choose a class $a_{i}$ that ends up being of the same size as one of the classes tagged with $y$ for some $y<x$, then computing $h(a_{i})$ for some isomorphism $h:A\to B$ might provide no information regarding $s_{x}$, as such classes could have relatively small indices compared to $s_{x}$.

As such, rather than finding the precise witness for $R_{x}$, we instead give a suitable range of potential witnesses and apply a pigeonhole argument to show that at least one of them is an actual witness. To be more precise, the desired property is for the image of the witness to not map to any of the classes with `small' indices; those which have been tagged either with $y$ or $y^{\dagger}$ for some $y<x$. It shall become evident that the number of such classes, say $n_{x}$, can be computed primitive recursively from $s_{y}$ for each $y<x$ (for now, we at least have $n_{1}=s_{0}+1$). By computing a given isomorphism $h:A\to B$ on the first $n_{x}$ many classes that are not potential witnesses for $R_{y}$ where $y<x$, at least one of them must map to a class in $B$ not tagged with $y$ or $y^{\dagger}$ for any $y<x$. We shall call these $n_{x}$ many classes the potential witnesses for $R_{x}$. Evidently, it follows that every class in $A$ is a potential witness for exactly one requirement $R_{x}$.

Whenever $g(x,s)\neq g(x,s-1)$, for each class $b_{i}$ where $i<s$ and $b_{i}$ is not tagged with $y$ or $y^{\dagger}$ for any $y<x$, tag $b_{i}$ with $x$. As per convention, we may `undo' any action performed for a lower priority $R_{z}$ by removing all tags $z$ or $z^{\dagger}$ for any $z>x$, and letting $p(b)\uparrow$ for any untagged $b$. Recall that the idea is for classes tagged with $x$ to avoid being the size of the witness for $R_{x}$. Since we now have multiple potential witnesses, each of these classes tagged with $x$ needs to have sizes larger than all sizes of potential witnesses for $R_{x}$. For each class $b_{i}$ tagged with $x$, define $p$ of these classes to be $a_{j}$ for the least $j$ satisfying all of the following.
\begin{itemize}
    \item $a_{j}$ has size larger than all the potential witnesses for $R_{x}$.
    \item $a_{j}$ has size at least as large as the current size of $b_{i}$.
    \item $a_{j}$ is currently not in the range of $p$.
\end{itemize}
For classes $b_{i}$ where $i\geq s$, let $p(b_{i})=a_{j}$, where $j$ is the least such that $a_{j}$ is not currently in the range of $p$, and $j$ is an index for one of the current potential witnesses of $R_{x}$. By pigeonhole principle, there must be at least one such $i\geq s$ for which such a $a_{j}$ can be found. Tag each $b_{i}$ for which $a_{j}$ can be found with $x^{\dagger}$. As before, let the size of $b$ be the same as $p(b)$ for each subsequent stage.

Once again, since $g^{*}$ is assumed to be total, $s_{x}$ must exist, and after such a stage, no new classes will ever be tagged with $x$. For each class $b_{i}$ tagged with $x$, whenever the sizes of one of the potential witnesses for $R_{x}$ increases, $p(b_{i})$ might need to be redefined. However, as the number of potential witnesses is fixed (depending only on $R_{y}$ for each $y<x$), there is some finite stage after which none of these potential witnesses ever change again. This implies that each of $p(b_{i})$ also eventually stabilises.

\

\emph{Ensuring $A,B\cong E$:} From the description above, it is evident that $A$ is isomorphic to $E$; $a_{i}$ copies $e_{i}$ once $e_{i}$ is enumerated. We shall verify formally that $p:B\to A$ is an isomorphism later. For now, we give a brief description of why this might be true. Recall that each $a_{i}$ is a potential witness for some $R_{x}$. Obviously, none of these classes will be the $p$ image for any class in $B$ tagged with $x$. Whilst some of these classes might be the $p$ image of classes in $B$ tagged with $y$ for some $y<x$, at least one of these potential witnesses must have a preimage which is not tagged with $y$. For such classes $a$, notice that we will always pick the potential preimage from $B$ as the one having the least index. Once such a class is found, $p(b)$ would have been defined to be $a$ and tagged with $x^{\dagger}$. Assuming that all the tags have stabilised, $a$ will forever remain in the range of $p$.

\

\emph{Construction:} At stage $s$, enumerate new classes $a_{s}$ into $A$ and $b_{s}$ into $B$, temporarily leaving them at size $1$.
\begin{description}
    \item[Step 1] Pick the least $x<s$ such that $g(x,s)\neq g(x,s-1)$. If no such $x$ exists, then proceed to the next step. Otherwise, for each $i<s$, such that $b_{i}$ is not tagged with $y$ or $y^{\dagger}$ for any $y<x$, tag $b_{i}$ with $x$.

    \item[Step 2] Define $a_{0}$ to be the (potential) witness for $R_{0}$. Suppose recursively that $a_{j}$ for each $j<i$ have all been defined to be potential witnesses for some $R_{y}$. Let $x$ be the least such that the potential witnesses for $R_{x}$ has yet to be defined. For each $k<n_{x}+1$ where $n_{x}$ is the number of classes in $B$ currently tagged with $y$ or $y^{\dagger}$ for some $y<x$, define $a_{i+k}$ to be a potential witness for $R_{x}$. Note that we renew the definition of the potential witnesses at each stage $s$.

    \item[Step 3] Proceeding in order of $i\leq s$, define $p_{s}(b_{i})$ to be $a_{j}$ for the least $j$ such that all of the following holds.
    \begin{itemize}
        \item $a_{j}$ is not yet in the range of $p_{s}$ and has size at least as big as the current size of $b_{i}$.

        \item If $b_{i}$ is tagged with $y$, then $a_{j}$ must have size larger than any potential witness for $R_{y}$.

        \item If $b_{i}$ is tagged with $y^{\dagger}$, then let $p_{s}(b_{i})=p_{s-1}(b_{i})$.
    \end{itemize}
    In addition, if $p_{s}(b_{i})$ is defined to be a potential witness for $R_{x}$ and is not tagged with $y$ for any $y<x$, then we tag $b_{i}$ with $x^{\dagger}$. If there is some $b_{i}$ for which $p_{s}(b_{i})$ cannot be defined, then we leave $p_{s}(b_{j})$ undefined for each $j\geq i$ and proceed to the next step.

    \item[Step 4] For each class $a_{i}$, let the size of $a_{i}$ be the same as the size of $e_{i}$. For each class $b_{i}$, let the size of $b_{i}$ be the same as the size of $p_{s}(b_{i})$ if defined. Otherwise, we do not change the size of $b_{i}$ at this stage.
\end{description}
Once Step 4 is completed, proceed to stage $s+1$.

\

\emph{Verification:} First, we show that the tags of each $b_{i}$ eventually stabilises. It is evident that every class $b_{i}$ must become tagged. Furthermore, once it is tagged, it never becomes untagged. A careful analysis of the construction will allow one to conclude that for a fixed class $b_{i}$, its tag could possibly go through the following changes.
\begin{itemize}
    \item If $b_{i}$ is currently tagged with $x^{\dagger}$, then its tag might later change to either $x$, $y$ or $y^{\dagger}$ for some $y<x$.

    \item If $b_{i}$ is currently tagged with $x$, then its tag might later change to either $y$ or $y^{\dagger}$ for some $y<x$.
\end{itemize}
It follows immediately that the tag of each $b_{i}$ eventually stabilises.

Since the tag of each $b_{i}$ eventually stabilises, the number of potential witnesses in $A$ for each $R_{x}$ also stabilises (see Step 2 of the construction for the definition). Recall that $a_{i}$ has the same (finite) size as $e_{i}$ for each $i$, there must be some finite stage after which the sizes of all potential witnesses of $R_{x}$ never again changes. By assumption that $E$ has unbounded character, $E$ must eventually produce classes with strictly larger size than all of these potential witnesses for $R_{x}$. Once such classes are discovered, Step 3 of the construction will define them to be the $p_{s}$ image of the classes $b_{i}$ tagged with $x$. Furthermore, as the sizes of the potential witnesses and the tags of $b_{i}$ no longer change, $p_{s}(b_{i})$ will also remain constant after this stage. For classes $b_{i}$ tagged with $x^{\dagger}$, we also have that provided its tag does not change, $p_{s}(b_{i})$ is always defined to be $p_{s-1}(b_{i})$.

We are now ready to show that $p=\lim_{s}p_{s}$ is an isomorphism. Since each $b_{i}$ copies $p_{s}(b_{i})$ once it is defined, then we will always have that $p(b_{i})$ and $b_{i}$ must be of the same size. In addition, $p_{s}(b_{i})$ is never defined to be $p_{s}(b_{j})$ for any $j<i$, and thus, $p$ is injective. It remains to show that $p$ is surjective.

Let $a_{j}\in A$ be given. Suppose also that $a_{j}$ is a potential witness for $R_{x}$, and that $a_{j}\neq p(b_{i})$ for any $b_{i}$ tagged with some $y<x$. We may further assume that we are at some stage where all of the following holds.
\begin{itemize}
    \item $p_{s}(b_{i})$ has stabilised for all $b_{i}$ tagged with either some $y\leq x$ or $y^{\dagger}$ for some $y<x$.
    \item $a_{j}$ is currently the class with the least index that is not yet in the range of $p_{s}$.
\end{itemize}
At such a stage, let $b_{i}$ be the class in $B$ with the least index such that $p_{s}(b_{i})\uparrow$ and $b_{i}$ has size at most the current size of $a_{j}$. If there is no $b_{k}$ where $k<i$ for which $p_{s}(b_{k})\uparrow$, then we must have defined $p_{s}(b_{i})=a_{j}$. Once $p_{s}(b_{i})$ is defined, it never again changes. This is because $b_{i}$ either has tag $y\leq x$ or $y^{\dagger}$ for some $y<x$, which is assumed to have stabilised, or $b_{i}$ will become tagged with $x^{\dagger}$ which cannot be replaced by any other tag at such a stage. On the other hand, if there is some $b_{k}$ such that $k<i$ and $p_{s}(b_{k})\uparrow$, then by the construction, the size of $b_{i}$ never increases until $p_{s}(b_{k})\downarrow$. In other words, $b_{i}$ will always remain at a size less than the current size of $a_{j}$. Once $p_{s}(b_{k})$ have stabilised for each $k<i$, $p_{s}(b_{i})$ must then be defined as $a_{j}$ and never again change.

Finally, to recover $g^{*}(x)$ for each $x$, define
$$
\Psi^{h}(x)=\begin{cases}
    0,&\text{if }x=0,\\
    \max\{h(a_{i})\mid i\leq\Psi^{h}(x-1)\}+\Psi^{h}(x-1)+1,&\text{otherwise.}
\end{cases}
$$
It is evident that $\Psi$ is a primitive recursive scheme. Given $x$, compute $g(x,\max\{h(a_{i})\mid i\leq\Psi^{h}(x)\})$. We claim that this must be equal to $g^{*}(x)$. The idea is that $\Psi^{h}(x)$ is an upper bound for the indices of potential witnesses for $R_{y}$ where $y\leq x$, and that at least one of the potential witnesses must be mapped to an index large enough to recover both the stabilising stage $s_{x}$ and the number of potential witnesses for $R_{x+1}$. Formally, we prove that the following statements hold.
\begin{enumerate}[label=(\roman*)]
    \item There is at least one class with tag $x^{\dagger}$, and at least one of $h(a)$, where $a$ is a potential witness for $R_{x}$, has an index which bounds the indices of any class with tag $x^{\dagger}$.

    \item If $a_{i}$ is a potential witness for $R_{y}$ for some $y\leq x$, then $i\leq\Psi^{h}(x)$.
\end{enumerate}
For any potential witness $a$ for $R_{x}$, $h(a)$ obviously cannot be tagged with $x$ as all such classes have size larger than the size of $a$. It is perhaps tedious, but not difficult to show that the number of classes which are tagged with $x^{\dagger}$, say $n$, are exactly the number of potential witnesses for $R_{x}$, such that given any isomorphism from $A\to B$, at least $n$ many potential witnesses for $R_{x}$ is not mapped to some class tagged with $y$ or $y^{\dagger}$ for some $y<x$. Applying this with the fact that there are more potential witnesses for $R_{x}$ than there are classes tagged with $y$ or $y^{\dagger}$ for some $y<x$ allows us to conclude that $n\geq 1$. That is, there is at least one class in $B$ tagged with $x^{\dagger}$. From the construction, we also have that $p^{-1}$ is a `minimal' isomorphism in the sense that for each $a_{i}$, $p^{-1}(a_{i})$ has the least possible index $k$ such that $p^{-1}(a_{j})\neq b_{k}$ for any $j<i$ and $b_{k}$ has the same size as $a_{i}$. Furthermore, if such a class $b_{k}$ is not tagged with $y$ or $y^{\dagger}$ for any $y<x$, it must be tagged with $x^{\dagger}$. In other words, in order for $h:A\to B$ to be an isomorphism, at least one of $h(a)$ must have index which bounds all indices of classes with tag $x^{\dagger}$.

Finally, we proceed via induction to show that (ii) also holds. The base case is trivial; the only (potential) witness for $R_{0}$ is $a_{0}$. Inductively suppose that (ii) holds for all $y<x$. Applying (i) allows us to obtain that $\max\{h(a_{i})\mid i\leq\Psi^{h}(x-1)\}$ is an index $k$ large enough such that any class in $B$ tagged with $y^{\dagger}$ for some $y<x$ has index $\leq k$. In addition, since any class with tag $y$ has index less than any class with tag $y^{\dagger}$, it must be that there are at most $k+1$ many potential witnesses for $R_{x}$. By the inductive hypothesis for $x-1$, we also have that if $a_{i}$ is a potential witness for $R_{y}$ for some $y\leq x-1$, then $i\leq\Psi^{h}(x-1)$. Therefore, $\Psi^{h}(x)=k+1+\Psi^{h}(x-1)$ satisfies (ii). Lemma~\ref{lem:d2fininf} and Theorem~\ref{thm:d2equiv} are proved.
\end{proof}

\section{Linear Orders}\label{sec:lin}

Another well-studied class in computable structure theory is that of linear orders. In this section, we study computably categorical linear orders, relatively $\Delta_{2}^{0}$-categorical linear orders, and some $\Delta_{3}^{0}$-categorical linear orders. Before we discuss the results and proofs, we introduce some notation that shall be used.
\begin{notation}
For linear orders $L_{0},L_{1}$:
\begin{itemize}
    \item The linear order $L_{0}+L_{1}$ is so that $a<b$ iff
    \begin{itemize}
        \item $a,b\in L_{0}$ or $a,b\in L_{1}$, and $a<b$ in $L_{0}$ or $L_{1}$ respectively;
        \item or $a\in L_{0}$ and $b\in L_{1}$.
    \end{itemize}
    \item The linear order $L_{0}*L_{1}$ is such that each point $a\in L_{1}$ is replaced with a copy of $L_{0}$.
    \item $L_{0}^{*}$ is given by the reversed ordering of $L_{0}$.
\end{itemize}
Finally, for a family $F$ of linear orders, $\shuffle(F)$ is obtained by replacing each point in $\eta$ with a copy of a member of $F$ so that between any two members of $F$ in $\shuffle(F)$, there exists a copy of each member of $F$. 
\end{notation}

\begin{theorem}[Theorem~2 in~\cite{bk21}]\label{thm:d1lin}
    If $L$ is a relatively $\Delta_{1}^{0}$-categorical linear order and not punctually categorical, then $\prcat(L)=\cone(\Delta_{1}^{0})$.
\end{theorem}

\begin{proof}
For a linear order $L$ to be computably categorical and not punctually categorical, $L$ must be the finite sum of $\eta$ and finite linear orders \cite{gd80,remmel81}. Let $g$, a total computable function be given. We shall utilise only one copy of $\eta$ in our strategy to encode $g$. The rest of the linear order $L$ can be enumerated in a standard way. We thus assume that $L$ is simply just $\eta$ for the rest of the proof.

$A$ will be the standard presentation of $\eta$; there is a primitive recursive function which computes $a_{k}$ between $a_{i},a_{j}$ for any given $i,j$. For the enumeration of $B$, we do the following at stage $s$.
\begin{enumerate}
    \item Enumerate $b_{s}$ into $B$. In addition, if $s>0$, then define $b_{s-1}<b_{s}$. That is, we have a uniformly primitive recursive sequence $b_{0}<b_{1}<b_{2}<\dots$ which shall be used to encode $g$. Also enumerate a standard copy of $\eta$ to the left of $b_{0}$.

    \item If $x\leq s$ is such that $g(x)[s]\downarrow$, then we `densify' the interval $(b_{x},b_{x+1})$. That is, after stage $s$, we construct a standard copy of $\eta$ between $b_{x}$ and $b_{x+1}$.
\end{enumerate}
Since $g$ is assumed to be total, for each $x$, there must be some stage $s$ such that for all $s'\geq s,\,g(x)[s']\downarrow$. Therefore, for each $x$, the interval between $b_{x}$ and $b_{x+1}$ is isomorphic to $\eta$. $B$ will thus have order type $\eta+1+\eta+1+\eta+\dots\cong \eta$. In addition, observe that all indices of elements between $b_{x}$ and $b_{x+1}$ is larger than the stage $s$ for which $g(x)[s]\downarrow$.

Let $h:A\to B$ be an isomorphism. Given $x$, compute $h^{-1}(b_{x})$ and $h^{-1}(b_{x+1})$. Since $A$ is the standard copy of $\eta$, we may produce primitive recursively $a_{k}$ such that $h^{-1}(b_{x})<a_{k}<h^{-1}(b_{x+1})$. Finally, by computing $h(a_{k})$, we obtain that $b_{x}<h(a_{k})<b_{x+1}$, and therefore, $h(a_{k})$ possesses index at least as large as the stage where $g(x)$ first converges, allowing us to recover $g(x)$.
\end{proof}

\

\begin{theorem}\label{thm:d2lin}
    If $L$ is a relatively $\Delta_{2}^{0}$-categorical linear order which is not relatively $\Delta^0_1$-categorical, then $\prcat(L)=\cone(\Delta_{2}^{0})$.
\end{theorem}

\begin{proof}
Let $g^{*}$ be a total $\Delta_{2}^{0}$ function with a primitive recursive approximation $g(x,s)$. Applying the classification presented in \cite[Theorem~2.7]{mccoy03}, $L$ must contain at least one of $\omega,\omega^{*},\zeta$ or $n*\eta$ for some $n>1$ as a sub-interval.

First suppose that $L$ contains $n*\eta$ as a sub-interval for some $n>1$. To construct $L$, we enumerate all of $L$ except the sub-interval containing exactly $n*\eta$ in a standard way. We focus only on constructing punctual presentations of this sub-interval, $n*\eta$. Each $n$ successive elements shall be referred to as a $n$-block. $A$ will be the `standard' copy; there is a primitive recursive procedure to compute the successor of $a_{i}$ provided it is not the right-most element of an $n$-block. The $n$-blocks in $A$ are also produced in a standard fashion. At each stage, between two currently successive $n$-blocks, enumerate a new one between them. Also enumerate new $n$-blocks both to the left and to the right of the current left-most and right-most $n$-blocks in $A$ respectively. It is easy to see that $A\cong n*\eta$. As usual, $B$ will be a `delayed' copy to encode the stabilising stage $s_{x}$ for each $x$.

Fix `standard' elements $b_{x}$ for each $x\in\omega$. The intention is to use the successor of $b_{x}$ to encode $s_{x}$. At each stage $s$ such that $g(x,s)\neq g(x,s-1)$, enumerate $n-1$ many fresh elements and define them to temporarily be in the same $n$-block as $b_{x}$. More specifically, we \emph{split} the current $n$-block containing $b_{x}$ into two successive $n$-blocks. The first contains $b_{x}$ as the left-most element and $n-1$ many fresh elements, while the second contains the old $n-1$ elements in the same $n$-block as $b_{x}$ with a fresh element to turn it into a $n$-block. To ensure that $B\cong n*\eta$, we `densify' in the standard way. Since $g^{*}$ is assumed to be total, there must be a final stage at which the $n$-block containing $b_{x}$ is split. Furthermore, it is evident that the true successor of $b_{x}$ has index $>s_{x}$. Given any isomorphism $h:A\to B$, compute $g(x,h(a))$ where $a$ is the successor of $h^{-1}(b_{x})$ (recall that $a$ can be found primitive recursively).

The construction can be arranged as follows. At stage $s$, in both $A$ and $B$, enumerate a new $n$-block between any two currently adjacent $n$-blocks. Also enumerate new left-most and right-most $n$-blocks into both $A$ and $B$. In addition, if $g(x,s)\neq g(x,s-1)$, split the $n$-block containing $b_{x}$. It is clear that the construction works and that given any isomorphism $h:A\to B$, $h\oplus h^{-1}\geq_{PR}g^{*}$ as explained earlier.

Now suppose that $L$ does not contain $n*\eta$ as a sub-interval. Then it has to contain at least one of $\omega,\omega^{*}$ or $\zeta$. As before, we enumerate the rest of $L$ in the standard way and turn our attention to constructing only the sub-interval containing either $\omega,\omega^{*}$ or $L$. We only explain the proof for the case when $L=\omega$ as a modification to either $\omega^{*}$ or $\zeta$ is trivial.

Let $A$ be the standard copy; $a_{0}<a_{1}<a_{2}\dots$. In $B$, we once again fix a standard increasing chain $b_{0}<b_{1}<\dots$. Whenever $g(x,s)\neq g(x,s-1)$, enumerate a new element directly adjacent to $b_{x}$. Given any isomorphism $h:A\to B$, and any $x\in\omega$, compute $g(x,h(a))$ where $a$ is the successor of $h^{-1}(b_{x})$. This value is evidently equal to $g^{*}(x)$. 
\end{proof}

\

\begin{theorem}\label{thm:omteta}
    $\prcat(\omega*\eta)=\cone(\Delta_{3}^{0})$.
\end{theorem}

\begin{proof}
It is known that $\omega*\eta$ as a linear order is relatively $\Delta_{3}^{0}$-categorical. By Fact \ref{fact:conesubspec}, we obtain that $\cone(\Delta_{3}^{0})\subseteq\prcat(\omega*\eta)$. The overarching idea shall be as follows. Given some primitive recursive approximation $g(x,s,t)$ to a total $\Delta_{3}^{0}$ function $g^{*}$, we have to encode for each $x$ and $s$, the stages $s_{x}$ and $t_{x,s}$ with the following properties.
\begin{itemize}
    \item For all $t\geq t_{x,s}$, $g(x,s,t)=g(x,s,t_{x,s})$.

    \item For all $s\geq s_{x}$, $\lim_{t}g(x,s_{x},t)=\lim_{t}g(x,s,t)$.
\end{itemize}
As usual, we construct punctual $A,B\cong\omega*\eta$ such that any isomorphism $h:A\to B$ is such that $h\oplus h^{-1}\geq_{PR}g^{*}$, by ensuring that we may primitive recursively recover $s_{x}$ and $t_{x,s_{x}}$ for each $x$ using $h\oplus h^{-1}$.

$A$ will be constructed as the `standard' copy. Each $\omega$-chain will be encoded via $a_{i,0}<a_{i,1}<a_{i,2}<\dots$. We shall refer to this $\omega$-chain as the $i^{th}$ $\omega$-chain. Furthermore, there is a primitive recursive procedure which given any two $i,j$, produces a $k$ such that $k^{th}$ $\omega$-chain is between the $i^{th}$ and $j^{th}$ $\omega$-chain. $B$ will be constructed as countably many sub-intervals, denoted by $B_{i}$ for each $i\in\omega$. $B_{0}$ will be isomorphic to $\omega*\eta$ and each subsequent $B_{i}$ will be isomorphic to $\omega+\omega*\eta$. Then the resulting structure $B=B_{0}+B_{1}+\dots$, will also be isomorphic to $\omega*\eta$. We use $B_{0}$ for the strategy to encode $t_{x,s}$ for each $x,s$, and $B_{x+1}$ for the strategy to encode $s_{x}$.

\

\emph{Encoding $t_{x,s}$:} Within $B_{0}$, fix the elements $b_{x,s}^{0}$ for each $x,s$. At the beginning of the construction, each of these elements will be the least element of their respective $\omega$-chains. Whenever it is discovered that $g(x,s,t)\neq g(x,s,t-1)$, we enumerate a new successor for $b_{x,s}^{0}$. Since $\lim_{t}g(x,s,t)$ is assumed to exist for each $x,s$, there must be some finite stage after which the value of $g(x,s,t)$ no longer changes. That is, we only change the successor of $b_{x,s}^{0}$ finitely often, and as long as we always extend the chain to the right, this process constructs an $\omega$-chain. Furthermore, the final successor of $b_{x,s}^{0}$ clearly has index $\geq t_{x,s}$. To recover this index, simply compute $h:A\to B$, a given isomorphism, on the element $a_{i,1}$ where $i$ is such that $h^{-1}(b_{x,s}^{0})=a_{i,0}$ (note that $b_{x,s}^{0}$ is the left-most element of an $\omega$-chain in $B$). Since $A$ is the standard copy, and by assumption that $h$ is an isomorphism, $h(a_{i,1})$ must be the successor of $b_{x,s}^{0}$.

\

\emph{Encoding $s_{x}$:} Let $P(x,s)$ denote the predicate ``for all $s'\geq s$, $\lim_{t}g(x,s,t)=\lim_{t}g(x,s',t)$''. Since $g$ is primitive recursive, this is a $\Pi_{2}^{0}$ predicate. Just as before, we may also assume that for each $x$, there is a unique $s$ for which $P(x,s)$ holds. During stages at which $P(x,s)$ look to be true, we say that $P(x,s)$ \emph{fires}. Since $P(x,s)$ is $\Pi_{2}^{0}$, it fires infinitely often iff $P(x,s)$ is true. We begin building $B_{x+1}$ by enumerating elements $b_{i,0}^{x+1}$ for each $i\in\omega$. We order these elements as an $\omega$-chain in order of their index $i$, and refer to them as the $i^{th}$ chain.

Whenever $P(x,s)$ fires, for each $n\geq s$, extend the $n^{th}$ chain such that it contains at least $s$ many elements. In addition, grow the sub-interval between the $n^{th}$ and $n+1$-th chain up to stage $s$ with the standard enumeration of $\omega*\eta$. The intention here is that should $P(x,s)$ fire infinitely often, for each $n<s$, there are only finitely many elements between $b_{n}^{x+1}$ and $b_{n+1}^{x+1}$, whilst the sub-interval between $n^{th}$ and $n+1$-th chain will be isomorphic to $\omega*\eta$ for each $n\geq s$.

Finally, to recover $s_{x}$, compute the given isomorphism $h:A\to B$ on the least element $a\in A$ of an $\omega$-chain between $h^{-1}(b_{0}^{x+1})$ and $h^{-1}(b_{0}^{x+2})$. If $h$ is an isomorphism, then $h(a)$ must be strictly between $b_{0}^{x+1}$ and $b_{0}^{x+2}$. In fact, $h(a)$ cannot even be in the same $\omega$-chain as $b_{0}^{x+1}$. But such an element should have index $\geq s_{x}$, as the first $b_{s}^{x+1}$ where there are infinitely many elements between $b_{s}^{x+1}$ and $b_{0}^{x+1}$ must be such that $s\geq s_{x}$.

\

\emph{Construction:} Since $A$ is the standard enumeration, we focus only on $B$. At each stage $n$, do the following.
\begin{description}
    \item[Step 1] For each $\langle x,s\rangle\leq n$ such that $P(x,s)$ fires, enumerate a standard copy of $\omega*\eta$ up to stage $n$ between the $m^{th}$ and $m+1$-th chain in $B_{x+1}$ for each $s\leq m\leq n$. Additionally, add a new right-most element to the $m^{th}$ chain for each $s\leq m\leq n$.

    \item[Step 2] For each $\langle x,s\rangle\leq n$ such that $g(x,s,n)\neq g(x,s,n-1)$, enumerate a new element and define it to be the new successor of $b_{x,s}^{0}$.

    \item[Step 3] For each $\langle x,s\rangle\leq n$, add a new right-most element to the chain currently containing $b_{x,s}^{0}$.
\end{description}

\

\emph{Verification:} We prove the following statements.
\begin{enumerate}[label=(\roman*)]
    \item For each $x\in\omega$, $B_{x+1}\cong\omega+\omega*\eta$. Furthermore, if $P(x,s)$ fires infinitely often, then the sum of the first $s+1$ many chains in $B_{x+1}$ forms a single $\omega$-chain.

    \item $B_{0}\cong\omega*\eta$ and for each $x,s$, the least element of the chain containing $b_{x,s}^{0}$ has index $\geq t_{x,s}$.
\end{enumerate}
Let $x\in\omega$ be given, and let $s_{x}$ be the unique value for which $P(x,s_{x})$ fires infinitely often. By Step 1 of the construction, we thus obtain that for each $m\geq s_{x}$, there are infinitely many stages $n$ during which we introduce a new right-most element to the $m^{th}$ chain. Therefore, for each $m\geq s_{x}$, the $m^{th}$ chain is an $\omega$-chain. In a similar vein, we also obtain that there are infinitely many stages $n$ such that we enumerate a standard copy of $\omega*\eta$ up to stage $n$ between the $m^{th}$ and $m+1$-th chain in $B_{x+1}$ for each $m\geq s_{x}$. Conversely, for each $m<s_{x}$, there is some finite stage after which we never again extend the length of the $m^{th}$ chain, and between the $m^{th}$ and $m+1$-th chain, only a finite part of $\omega*\eta$ is enumerated. That is to say, there are only finitely many elements to the left of the $s_{x}^{th}$ chain in $B_{x+1}$ and all subsequent chains are $\omega$-chains with a copy of $\omega*\eta$ between them. We thus obtain that $B_{x+1}\cong\omega+\omega*\eta+\omega+\omega*\eta+\dots\cong\omega+\omega*\eta$. By convention, any element contained within the $\omega*\eta$ part of $B_{x+1}$ will have index $>s_{x}$, since the $s_{x}^{th}$ chain of $B_{x+1}$ is the tail of the first $\omega$-chain of $B_{x+1}$.

For (ii), fix some $x,s\in\omega$. By Step 2 of the construction, whenever $g(x,s,n)\neq g(x,s,n-1)$, $b_{x,s}^{0}$ gets a new successor. Since $\lim_{t}g(x,s,t)$ is assumed to exist, there is some finite stage after which we never again enumerate a new successor for $b_{x,s}^{0}$. On the other hand, by Step 3 of the construction, there are infinitely many stages during which we extend the chain containing $b_{x,s}^{0}$ to the right. Thus, each such chain will be an $\omega$-chain at the end of the construction. By a careful arrangement of these chains, we may obtain that $B_{0}\cong\omega*\eta$. Furthermore, we also have that the final successor of $b_{x,s}^{0}$ has index $\geq t_{x,s}$, as we always enumerate a new one whenever the approximation $g(x,s,n)$ changes.

Finally, let $h:A\to B$ an isomorphism be given. For any $x\in\omega$, first compute $h^{-1}(b_{0}^{x+1})$ and $h^{-1}(b_{0}^{x+2})$. Since $b_{0}^{x+1}$ and $b_{0}^{x+2}$ are the left-most elements of $B_{x+1}$ and $B_{x+2}$ respectively, they are in distinct $\omega$-chains within $B$. Using the fact that $A$ is a standard enumeration, we may primitive recursively obtain some element $a$ strictly between $h^{-1}(b_{0}^{x+1})$ and $h^{-1}(b_{0}^{x+2})$. Now, computing $h(a)$ allows us to obtain some index $s^{*}\geq s_{x}$, as $h(a)$ must be contained in the copy of $\omega*\eta$ in $B_{x+1}$. Once such an $s^{*}$ is obtained, compute $h(a')$ where $a'$ is the successor of $h^{-1}(b_{x,s^{*}}^{0})$. As before, since $A$ is standard, such an $a'$ can be found primitive recursively from $h^{-1}(b_{x,s^{*}}^{0})$. By the argument above, $h(a')$ must have index $t^{*}\geq t_{x,s^{*}}$. It follows that $g(x,s^{*},t^{*})=\lim_{t}g(x,s^{*},t)=\lim_{s}\lim_{t}g(x,s,t)=g^{*}(x)$. Theorem~\ref{thm:omteta} is proved.
\end{proof}

\

\begin{remark}\label{rem:shuffle}
Observe that if $L=\shuffle(\{\omega\}\cup F)$, then $L\cong L+\omega+L+\omega+\dots$. In other words, with the assumption that $L$ is computably presentable (and hence punctually presentable by Theorem \ref{thm:puncpres}) we may obtain that $\prcat(L)\subseteq\cone(\Delta_{3}^{0})$, by replacing $\omega*\eta$ with $L$ in the proof of Theorem \ref{thm:omteta}.
\end{remark}

\begin{corollary}\label{cor:shuffle}
    Let $F$ be a computable family of linear orders. If $L=\shuffle(\{\omega\}\cup F)$ is $\Delta_{3}^{0}$-categorical, then $\prcat(L)=\cone(\Delta_{3}^{0})$.
\end{corollary}

Note that a key property of the linear order used in the proof of Theorem \ref{thm:omteta} is that we can `hide' infinitely many elements with `small' indices in a single $\omega$ chain. This ensures that we are able to create sub-intervals with only elements of sufficiently large index. For example, the copy of $\omega*\eta$ in each $B_{x+1}$ in the proof of Theorem \ref{thm:omteta} only consists of elements with indices larger than some stabilising stage $s_{x}$. Thus, in considering linear orders of the form $\shuffle(F)$ for some family $F$ of linear orders, we suspect that the proof for the case where $F$ contains some infinite linear order (with at least two adjacent elements) will be similar to the proof presented earlier. As such, we turn our attention to the case where $F$ contains no infinite linear orders.

\begin{theorem}\label{thm:shuffle}
    Let $L$ be a linear order of the form $\shuffle(F)$, where $|F|>1$ and $F$ is a computable family of finite linear orders, then $\prcat(L)=\cone(\Delta_{3}^{0})$.
\end{theorem}

\begin{proof}
Fix some computable family of finite linear orders $F$, and let $M<N$ be the two smallest sizes of linear orders contained in $F$. We refer to these as $M$-chains and $N$-chains. Once again, given a total $\Delta_{3}^{0}$ function $g^{*}$, the goal is to encode the stages $s_{x},t_{x,s}$ such that, for each $s\geq s_{x},\,\lim_{t}g(x,s_{x},t)=\lim_{t}g(x,s,t)$, and for each $t\geq t_{x,s},\,g(x,s,t)=g(x,s,t_{x,s})$, where $g$ is a given primitive recursive approximation to $g^{*}$. Let $A$ be the `standard' enumeration of $\shuffle(F)$; there are primitive recursive procedures that do the following.
\begin{itemize}
    \item Given any element $a$, we may discover if $a$ is in a $M$-chain or $N$-chain (or some special symbol if neither holds).

    \item Given any two elements $a,a'$, we may produce the indices of the elements in a $M$-chain (similarly $N$-chain) strictly between $a$ and $a'$.
\end{itemize}
$B$ will be constructed in a similar way as before. Let $B=B_{0}+B_{1}+\dots$, where $B_{0}\cong \shuffle(F)$, and each subsequent $B_{x}\cong M+\shuffle(F)$. We use $B_{0}$ to encode $t_{x,s}$ for each $x,s\in\omega$, and $B_{x+1}$ to encode $s_{x}$ for each $x\in\omega$.

\

\emph{Encoding $t_{x,s}$:} Within $B_{0}$, fix elements $b_{x,s}$, which we keep always as the left-most element of some $N$-chain. Note that $N$ is at least $2$ since $N>M\geq 1$, that is, there will be at least one other element in the $N$-chain containing $b_{x,s}$. Evidently, the elements in the same $N$-chain as $b_{x,s}$ will be used to encode $t_{x,s}$. Whenever $g(x,s,t)\neq g(x,s,t-1)$, we split the current $N$-chain into two new $N$-chains, one containing $b_{x,s}$ with $N-1$ fresh elements (with indices $\geq t$), and the other containing a new element together with the remaining $N-1$ elements. Since $t_{x,s}$ must exist, $b_{x,s}$ will be the left-most element of some $N$-chain, where the other $N-1$ elements all have indices at least $t_{x,s}$. This can be easily recovered by computing $h(a)$ where $a$ is the successor of $h^{-1}(b_{x,s})$ given any isomorphism $h:A\to B$.

\

\emph{Encoding $s_{x}$:} Let $P(x,s)$ be the predicate ``for all $s'\geq s,\,\lim_{t}g(x,s,t)=\lim_{t}g(x,s',t)$'', and assume that for each $x$, there is exactly one $s$ for which $P(x,s)$ \emph{fires} infinitely often (looks to be true for infinitely many stages). The idea here is to ensure that all $M$-chains in $B_{x+1}$, save the left-most one, will have indices $\geq s_{x}$. Fix some enumeration of $\shuffle(F)$. Wherever we would have enumerated an $M$-chain, we instead enumerate a $N$-chain, with its first $M$ many elements \emph{tagged} with $P(x,s)$ for various $s\in\omega$. Whenever $P(x,s)$ fires, for each $s'\geq s$, split each $N$-chain containing an $M$-chain tagged with $P(x,s')$, into $N-M+1$ many new chains, one of which is the tagged $M$-chain, and the rest are singletons. After splitting, extend all singletons to new $N$-chains, and extend the tagged $M$-chain to a new $N$-chain. Each of these tagged $M$-chains will be in one of the following forms at the end of the construction.
\begin{itemize}
    \item In the event that $P(x,s)$ fires infinitely often, any $M$-chain tagged with $P(x,s')$ for $s'\geq s$ will have its `tail' removed infinitely often. Thus, they become $M$-chains in the limit of the construction. In addition, the `junk' produced from removing this `tail' are all turned into $N$-chains.

    \item If we instead have that $P(x,s')$ fires only finitely often for all $s'\leq s$, then after some finite stage, the `tail' of the $M$-chain tagged with $P(x,s)$ never again gets removed. Thus, it remains as an $N$-chain for till the end of the construction.
\end{itemize}
The idea is here is that only finitely many of these tagged chains are $N$-chains, while the rest are turned into $M$-chains. In particular, with some care in arranging how these tagged chains are enumerated, the resulting isomorphism type should still be $\shuffle(F)$. Finally, since the only $M$-chains formed by this procedure are exactly those tagged with $P(x,s')$ where $P(x,s)$ fires infinitely often for some $s\leq s'$, all $M$-chains (except the left-most one) in $B_{x+1}$ must have index $\geq s_{x}$. To recover this stage, let $b_{x+1},b_{x+2}$ be the left-most elements of $B_{x+1}$ and $B_{x+2}$ respectively. Then for any given isomorphism $h:A\to B$, we compute $h(a)$ where $a$ is an element of some $M$-chain strictly between $h^{-1}(b_{x+1})$ and $h^{-1}(b_{x+2})$. That is, $h(a)$ is an $M$-chain that lies strictly between $b_{x+1}$ and $b_{x+2}$, and therefore should have index $\geq s_{x}$.

\

\emph{Construction:} During the construction, we maintain queues $Q_{x}$ to dictate enumerations of the various $B_{x}$. At each stage $n$, do the following.
\begin{description}
    \item[Step 1] For each pair of chains currently in $B_{0}$, and for each member of $F$ currently missing between these chains, add an action to enumerate the missing chain into the queue $Q_{0}$. Then pick the first action in the queue and enumerate the required member of $F$. If we also have that it was an $N$-chain that was enumerated, then label the left-most element of this $N$-chain with $b_{x,s}$ where $\langle x,s\rangle$ is the least label that has yet to be used.

    \item[Step 2] If $g(x,s,n)\neq g(x,s,n-1)$, then perform the procedure as described in the strategy for encoding $t_{x,s}$.
    
    \item[Step 3] Enumerate the left-most $M$-chain of $B_{n+1}$ with a distinguished element $b_{n+1}$.
    
    \item[Step 4] For each $x\leq n$, for each pair of chains in $B_{x+1}$, and for each member of $F\setminus\{M\text{-chain}\}$ currently missing between these chains, add an action to enumerate the missing chain into the queue $Q_{x+1}$. Similarly, for each pair of chains in $B_{x+1}$ which do not yet contain any chain tagged with $P(x,s)$ for any $s\in\omega$ between them, we also add an action to enumerate the missing chain into $Q_{x+1}$. Pick the first action in $Q_{x+1}$ and enumerate the required chain into the corresponding sub-interval.

    \item[Step 5] If $P(x,s)$ fires, apply the procedure as described earlier to the chains tagged with $P(x,s')$ for each $s'$ where $s\leq s'\leq n$.
\end{description}

\

\emph{Verification:} We verify that $B\cong\shuffle(F)$ and that each $B_{x}$ has the desired properties described in the strategies. From Step 1 of the construction, it is evident that $B_{0}\cong\shuffle(F)$. For any two chains, there must be some stage at which at least one of each member of $F$ is enumerated between them. Also, it is evident that for each $x,s\in\omega$, there is an $N$-chain in $B_{0}$ with left-most element $b_{x,s}$, and that the successor of $b_{x,s}$ has index $\geq t_{x,s}$.

It remains to show that for each $x\in\omega$, $B_{x+1}\cong M+\shuffle(F)$. Following the description in the strategy for encoding $s_{x}$, there can only be finitely many tagged chains that are $N$-chains in the limit of the construction; these are exactly the chains tagged with $P(x,s')$ for some $s'<s$ where $s$ is the unique value for which $P(x,s)$ holds. We denote these chains in order by $C_{0},C_{1},\dots,C_{s-2}$. Applying the fact that all other tagged chains will become $M$-chains in the limit of the construction, it follows that Step 4 of the construction guarantees that the sub-interval between $C_{i}$ and $C_{i+1}$ is isomorphic to $\shuffle(F)$. A similar property should also hold for the sub-intervals in $B_{x+1}$ to the left of $C_{0}$ or to the right of $C_{s-2}$. In other words, together with the left-most $M$-chain in $B_{x+1}$, we have that $B_{x+1}\cong M\text{-chain}+\shuffle(F)+N\text{-chain}+\shuffle(F)+N\text{-chain}+\dots+\shuffle(F)$. Since $N$-chains are a member of the family $F$, $B_{x+1}\cong M+\shuffle(F)$.

To recover $g^{*}(x)$, for any given isomorphism $h:A\to B$, apply the following procedure. Compute $h(a)$ where $a$ is an element of an $M$-chain between $h^{-1}(b_{x+1})$ and $h^{-1}(b_{x+2})$. Then $h(a)$ must be contained in an $M$-chain between $b_{x+1}$ and $b_{x+2}$, which guarantees that it has some index $s^{*}\geq s_{x}$. For this index $s^{*}$, find the successor of $b_{x,s^{*}}$, by computing $h(a')$, where $a'$ is the successor of $h^{-1}(b_{x,s})$. This allows us to obtain an index $t^{*}\geq t_{x,s^{*}}$. Finally, compute $g(x,s^{*},t^{*})$, which must equal $g^{*}(x)$. Theorem~\ref{thm:shuffle} is proved.
\end{proof}

\

\begin{theorem}\label{thm:ompeta}
    $\prcat(\omega+\eta)=\cone(\Delta_{3}^{0})$.
\end{theorem}

\begin{proof}
As before, we use $s_{x}$ and $t_{x,s}$ to denote the stages such that $\lim_{t}g(x,s_{x},t)=\lim_{s}\lim_{t}g(x,s,t)$ and $g(x,s,t_{x,s})=\lim_{t}g(x,s,t)$, for a given primitive recursive approximation $g$ to a total $\Delta_{3}^{0}$ function $g^{*}$. Let $A$ be the standard copy of $\omega+\eta$; both the $\omega$ part and $\eta$ parts are enumerated in the standard way. Unlike the proofs of Theorem \ref{thm:omteta} and \ref{thm:shuffle}, we no longer have unique places to encode the various stabilising stages. We need to be able to `nest' the locations at which we use to encode the various $t_{x,s}$ and $s_{x}$. As such we take a slightly different approach from before.

\

\emph{Encoding $g^{*}(0)$:} There are two main things which we need to encode. First, we have to encode $s_{0}$, and second, for each $s\in\omega$, we must encode $t_{0,s}$. $B$ will start off as an $\omega$-chain, consisting of elements $b_{0}<b_{1}<\dots$. Whenever $g(0,s,t)\neq g(0,s,t-1)$, we enumerate a new successor for the element $b_{s}$. Obviously the intention here is to ensure that the successor of $b_{s}$ has index $\geq t_{0,s}$.

Let $P(x,s)$ be the $\Pi_{2}^{0}$ predicate ``for all $s'\geq s,\,\lim_{t}g(x,s,t)=\lim_{t}g(x,s',t)$'', and assume that for each $x$, there is a unique $s$ for which $P(x,s)$ holds. Whenever $P(0,s)$ fires, for each $s'>s$, `densify' the interval between $b_{s'}$ and $b_{s'+1}$. Since $P(0,s)$ holds iff $P(0,s)$ fires infinitely often, it is evident that every interval between $b_{s}$ and $b_{s+1}$ for $s>s_{0}$ is isomorphic to $\eta$. The idea is that $b_{s_{0}}$ should remain in the $\omega$ part of $B$.

To recover the value $g^{*}(0)$, for a given isomorphism $h:A\to B$, compute $h$ on some element within the $\eta$ part of $A$. Since $h$ is an isomorphism, this should map to the $\eta$ part of $B$. This should have index $s^{*}\geq s_{0}$. Once such an index is obtained, we compute $h^{-1}(b_{s})$ for each $s\leq s^{*}$. Since $A$ is standard, we may primitive recursively discover if $h^{-1}(b_{s})$ is in the $\omega$ or $\eta$ part of $A$. By the strategy described above, at least one of $h^{-1}(b_{s})$ should be contained in the $\eta$ part of $A$. The final $s$ for which $h^{-1}(b_{s})$ is contained in the $\omega$ part of $A$ necessarily has index exactly $s_{0}$. After $s_{0}$ is found, we again use $h^{-1}$ and $h$ to obtain the successor of $b_{s_{0}}$, which should have index $\geq t_{0,s_{0}}$, thus allowing us to compute $g^{*}(0)$.

\

\emph{Labelling elements:} Before we explain how we might nest the strategies to encode the various $g^{*}(x)$, we first introduce a system of labels for the elements of $B$. Each label is a string $\sigma\in\omega^{<\omega}$, and if it has length $x+1$, then it is enumerated for the sake of encoding $g^{*}(x)$. As explained earlier, the elements $b_{0},b_{1},\dots$ are for the sake of encoding $g^{*}(0)$. We induce a well-order on the labels (and thus the elements) by letting $\sigma\leq\tau$ iff $\sigma$ is a prefix of $\tau$ or $\sigma$ is lexicographically left of $\tau$.

\

\emph{Encoding $g^{*}(x)$:} To nest the strategies, we define \emph{active} intervals recursively. Whenever $P(0,s)$ fires, all intervals to the left of $b_{s}$ is declared to be inactive for $x=1$, and all intervals to the right of $b_{s}$ will be declared active for $x=1$. These definitions remain until the next $P(0,s')$ fires, after which we renew the definitions of active and inactive (for $x=1$). The idea here is that we only apply the strategy for $g^{*}(x)$ in intervals that have been declared active for $x$.

Let $\sigma$ be such that the interval to the left of $b_{\sigma}$ has been declared inactive for $x$, and the interval to the right of $b_{\sigma}$ has been declared active for $x$. For each $\tau\geq\sigma$ (recall the ordering defined on the labels earlier), we begin enumerating the elements $b_{\tau^{\frown}0},b_{\tau^{\frown}1},\dots$ for the sake of encoding $g^{*}(x)$. Whenever $g(x,s,t)\neq g(x,s,t-1)$, enumerate a new successor to the element $b_{\tau^{\frown}s}$ with index $\geq t$. If $P(x,s)$ fires, then for each $s'>s$, `densify' the interval between $b_{\tau^{\frown}s'}$ and $b_{\tau^{\frown}(s'+1)}$. In addition, declare the interval to the left of $b_{\sigma^{\frown}s}$ inactive for $x+1$, and the interval to the right of $b_{\sigma^{\frown}s}$ active for $x+1$ (see Fig.~\ref{fig:ompeta} for an example).
\begin{figure}
    \centering
    \begin{tikzpicture}
    \node at (0,0) {$\dots$};
    \node at (1,0) [label={\tiny $b_{s}$}] {$\bullet$};
    \node at (7,0) [label={\tiny $b_{s+1}$}] {$\bullet$};
    \node at (11,0) [label={\tiny $b_{s+2}$}] {$\bullet$};
    \node at (12,0) {$\dots$};

    \foreach \x/\y in {2/0,3/1,5/t}
    {
    \node at (\x,0) [label={\tiny $b_{s,\y}$}] {$\bullet$};
    }
    \node at (4,0) {$\dots$};
    \node at (6,0) {$\dots$};
    
    \foreach \x/\y in {8/0,9/1}
    {
    \node at (\x,0) [label={\tiny $b_{s+1,\y}$}] {$\bullet$};
    }
    \node at (10,0) {$\dots$};
    
    \draw[<->] (5.1,-0.5) --node[below,pos=0.5]{\tiny Active for $2$} (12.5,-0.5);
    \draw[<->,dashed] (-0.5,-0.5) --node[below,pos=0.5]{\tiny Inactive for $2$} (4.9,-0.5);

    \draw[<->] (1.1,1) --node[above,pos=0.5]{\tiny Active for $1$} (12.5,1);
    \draw[<->,dashed] (-0.5,1) --node[above,pos=0.5]{\tiny Inactive for $1$} (0.9,1);
    
    \end{tikzpicture}
    \caption{Active intervals if $P(0,s)$ and $P(1,t)$ fires.}
    \label{fig:ompeta}
\end{figure}
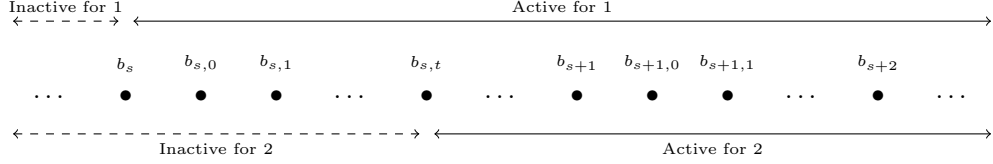

The rough idea here is that for each $y$, there will only be finitely many elements to the left of $b_{s_{0},s_{1},\dots,s_{y}}$, and that the interval to the right of $b_{s_{0},s_{1},\dots,s_{y}+1}$ is isomorphic to $\eta$. Furthermore, $b_{s_{0},s_{1},\dots,s_{y}}$ is contained in the $\omega$ part of $B$, and $b_{s_{0},s_{1},\dots,s_{y}+1}$ is contained in the $\eta$ part of $B$. In order to recover $g^{*}(x)$, assume recursively that $s_{0},s_{1},\dots,s_{x-1}$ have all been found. Given any isomorphism $h:A\to B$, computing $h(a)$ where $a$ is in the $\eta$ part of $A$, and to the left of $h^{-1}(b_{s_{0},s_{1},\dots,s_{x-1}+1})$, provides an upper bound on the index $s^{*}$, for which $b_{s_{0},s_{1},\dots,s_{x-1},s^{*}}$ is contained in the $\omega$ part of $B$. By computing $h^{-1}(b_{s_{0},s_{1},\dots,s_{x-1},s})$ for each $s\leq s^{*}$, we may discover exactly the largest $s$ such that $b_{s_{0},s_{1},\dots,s_{x-1},s}$ is contained in the $\omega$ part of $B$. By the construction, this $s$ should be exactly $s_{x}$. Once $s_{x}$ is found, we may easily recover an index $t^{*}\geq t_{x,s_{x}}$ by computing the successor of $b_{s_{0},s_{1},\dots,s_{x}}$ using $h,h^{-1}$.

It is perhaps not too surprising that the ideas above very closely resemble a $\emptyset''$-priority argument. In some ways, the `true path' computes the point of separation between the $\omega$ and $\eta$ parts of our structure $B$, which aligns with the main difficulty in finding an isomorphism between two presentations of $\omega+\eta$.

\

\emph{Construction:} During the construction, we maintain parameters $s_{x,n}$ for each $x\leq n$ defined to be the most recent value $s$ for which $P(x,s)$ has fired. Clearly, $\liminf_{n}s_{x,n}=s_{x}$, the value for which $\lim_{t}g(x,s_{x},t)=g^{*}(x)$. For notational convenience, let $\delta_{0,n}=\langle\rangle$ and $\delta_{x,n}\coloneqq\langle s_{0,n},s_{1,n},\dots,s_{x-1,n}\rangle$. At stage $n$ of the construction, do the following.
\begin{description}
    \item[Step 1] For each $x\leq n$, \emph{densify} the interval to the right of $b_{\delta_{x,n}\!^{\frown}(s_{x,n}+1)}$; for each pair of elements currently adjacent in the interval, enumerate a new element between them. By convention, we may assume that such an element has index larger than the pair of previously adjacent elements.

    \item[Step 2] For each $i,x\leq n$ and for each $\tau\geq\delta_{x,n}$ where $|\tau|=x$, enumerate the elements $b_{\tau^{\frown}i}$ into $B$ if they have yet to be enumerated. For a given $\sigma$ and $x$, when enumerating elements with labels $\sigma^{\frown}x^{\frown}0$, ensure that all earlier elements enumerated by Step 1 within the interval $b_{\sigma^{\frown}x}$ and $b_{\sigma^{\frown}(x+1)}$ is to the left of $b_{\sigma^{\frown}x^{\frown}0}$. The idea is to keep the indices of elements between $b_{\sigma^{\frown}x^{\frown}0}$ and $b_{\sigma^{\frown}(x+1)}$ (relatively) large. 
    
    \item[Step 3] If $g(x,s,n)\neq g(x,s,n-1)$, then for each $\tau\geq\delta_{x,n}$ such that $|\tau|=x$, enumerate a new successor for the element $b_{\tau^{\frown}s}$ provided such an element is currently in the structure.

    \item[Step 4] If $P(x,s)$ fires, then update $s_{x,n+1}=s$.
\end{description}

\

\emph{Verification:} Let $\delta_{x}$ be defined as $\langle s_{0},s_{1},\dots,s_{x-1}\rangle$, where $s_{y}$ is the value for which $\lim_{t}g(y,s_{y},t)=g^{*}(y)$. We prove that for each $x\in\omega$, there are only finitely many elements ever enumerated to the left of $b_{\delta_{x}\!^{\frown}s_{x}}$ and that the interval to the right of $b_{\delta_{x}\!^{\frown}(s_{x}+1)}$ is isomorphic to $\eta$. 

Let $n^{*}$ be some stage large enough such that $P(0,s)$ never again fires for any $s<s_{0}$. Such a stage must exist, as we assumed that there is exactly one $s$ for which $P(0,s)$ fires infinitely. From the construction, at each stage $n$, we only enumerate elements for the following reasons.
\begin{itemize}
    \item First, the element enumerated was for the sake of Step 1. These elements are enumerated to densify intervals to the right of $b_{\delta_{y,n}\!^{\frown}(s_{y,n}+1)}$. By the ordering defined on the labels, for any $y$, we always have that $\delta_{y,n}\!^{\frown}(s_{y,n}+1)\geq s_{0,n}\geq s_{0}$ (as labels) for any $n\geq n^{*}$. Thus, such elements are only enumerated to the right of $b_{s_{0}}$.
    
    \item Second, the element enumerated is of the form $b_{\tau^{\frown}i}$ where $\tau\geq\delta_{y,n}$ and $|\tau|=y$ (for the sake of Step 2). We may assume that $y>0$, as there are necessarily only finitely many elements of the form $b_{\delta_{0,n}\!^{\frown}i}$ to the left of $b_{s_{0}}$. Just as before, for $y>0$, we always have that $\delta_{y,n}\geq s_{0,n}\geq s_{0}$, for any $n\geq n^{*}$, and thus, $b_{\tau^{\frown}i}$ must be to the right of $b_{s_{0}}$.

    \item Finally, Step 3 possibly adds only finitely many more elements for each labelled element currently in the structure. Thus, after some finite stage, this stops affecting the finitely many labelled elements to the left of $b_{s_{0}}$ added by Step 2.
\end{itemize}
Thus, after some finite stage, no new elements are ever enumerated to the left of $b_{s_{0}}$. It is also easy to see that since $P(0,s_{0})$ fires infinitely often, there are infinitely many stages during which we densify the interval to the right of $b_{s_{0}+1}$. Therefore, the interval to the right of $b_{s_{0}+1}$ is isomorphic to $\eta$.

Now let $x>0$ be given. By induction, we have that there are only finitely many elements ever enumerated to the left of $b_{\delta_{x}}=b_{\delta_{x-1}\!^{\frown}s_{x-1}}$, and that the interval to the right of $b_{\delta_{x-1}\!^{\frown}(s_{x-1}+1)}$ is isomorphic to $\eta$. Fix a stage $n^{*}$ large enough such that after stage $n^{*}$, for any $y\leq x$, $P(y,s)$ never again fires for any $s<s_{y}$, and no more elements are enumerated to the left of $b_{\delta_{x}}$. We again analyse where elements are enumerated by the construction after such a stage.
\begin{itemize}
    \item Step 1 of the construction might densify intervals to the right of $b_{\delta_{y,n}\!^{\frown}(s_{y,n}+1)}$ for each $y\leq n$. Notice that $y$ is possible larger than $x$. By choice of $n^{*}$, we have that for each $n\geq n^{*}$ and for each $z\leq x,\,s_{z}\leq s_{z,n}<s_{z,n}+1$. Therefore, by the ordering defined on the intervals, $\delta_{y,n}\!^{\frown}(s_{y,n}+1)\geq\langle s_{0},s_{1},\dots,s_{y},\dots,s_{x}\rangle=\delta_{x}\!^{\frown}s_{x}$. After stage $n^{*}$, this action only causes elements to be enumerated to the right of $b_{\delta_{x}\!^{\frown}s_{x}}$.
    
    \item For Step 2 of the construction, we enumerate elements of the form $b_{\tau^{\frown}i}$ for each $i,y\leq n$ and for each $\tau\geq\delta_{y,n}$ where $|\tau|=y$. By choice of $n^{*}$, all such elements are to the right of $b_{\delta_{x}}$. Suppose for a contradiction that one of these elements is strictly between $b_{\delta_{x}}$ and $b_{\delta_{x}\!^{\frown}s_{x}}$. If $|\tau|=x$, then for $b_{\tau^{\frown}i}$ to be strictly between $b_{\delta_{x}}$ and $b_{\delta_{x}\!^{\frown}s_{x}}$, it must be that $i<s_{x}$. However, there are clearly only finitely many such elements, and we may simply pick a stage large enough such that all of these have already been enumerated. Thus, we may assume that $|\tau|>x$.

    If $|\tau|>x$ and $b_{\tau^{\frown}i}$ is strictly between $b_{\delta_{x}}$ and $b_{\delta_{x}\!^{\frown}s_{x}}$, then there must be some $n\geq n^{*}$ such that $\delta_{x}\leq\delta_{y,n}\leq\tau<\tau^{\frown}i\leq\delta_{x}\!^{\frown}s_{x}$. By the definition of the ordering on the labels, $\delta_{x}$ must be a prefix of $\delta_{y,n}$. Furthermore, as $|\delta_{y,n}|=y=|\tau|\geq x+1$, $\delta_{y,n}$ in fact contains $\delta_{x}\!^{\frown}s_{x,n}$ as a prefix. By choice of $n^{*}$, we know that for any $n\geq n^{*},\,s_{x,n}\geq s_{x}$, which is to say that $\delta_{y,n}\geq\delta_{x}\!^{\frown}s_{x}$, leading to a contradiction.

    \item For the same reasons as before, since we have that Step 2 only enumerates finitely many labelled elements to the left of $b_{\delta_{x}\!^{\frown}s_{x}}$, Step 3 possibly only contributes finitely many more elements to the left of $b_{\delta_{x}\!^{\frown}s_{x}}$.
\end{itemize}
Thus, there can only be finitely many elements ever enumerated to the left of $b_{\delta_{x}\!^{\frown}s_{x}}$. It is easy to see that since $P(y,s_{y})$ fires infinitely often for each $y\leq x$, then the interval to the right of $b_{\delta_{x}\!^{\frown}(s_{x}+1)}$ is densified infinitely often, resulting in an interval isomorphic to $\eta$.

Notice that by the ordering on the labels, we obtain that $I_{x+1}$ is always a strict sub-interval of $I_{x}$, where $I_{x}$ is defined to be the interval bounded by $b_{\delta_{x}\!^{\frown}s_{x}}$ and $b_{\delta_{x}\!^{\frown}(s_{x}+1)}$. Applying this with the fact that there are always only finitely many elements to the left of $I_{x}$, and the collection of elements to the right of $I_{x}$ is isomorphic to $\eta$, we obtain that $B\cong\omega+\eta$. Furthermore, by the careful enumeration of elements in Step 2 of the construction, we also have that any element in the $\eta$ part of $I_{x}$ has index $>s_{x+1}$.

Let $h:A\to B$ be an isomorphism, and let $x\in\omega$ be given. To compute $g^{*}(0)$, first compute $h(a)$ for an arbitrary element $a$ in the $\eta$ part of $A$. Clearly, $h(a)$ should also be contained in the $\eta$ part of $A$. That is, $h(a)$ should be to the right of elements of the form $b_{\delta_{y}\!^{\frown}(s_{y}+1)}$ for some $y\geq 0$. Any such element is necessarily to the right of $b_{s_{0}}$; $h(a)$ must have index $s^{*}>s_{0}$. By computing $h^{-1}(b_{s})$ for each $s\leq s^{*}$, we may recover $s_{0}$ as the final index for which $h^{-1}(b_{s})$ is still contained in the $\omega$ part of $A$. Once $s_{0}$ is found, compute $h(a^{*})$ where $a^{*}$ is the successor of $h^{-1}(b_{s_{0}})$. By Step 2 of the construction, such an element must have index $t^{*}\geq t_{0,s_{0}}$. Finally, we may then obtain $g^{*}(0)$ by computing $g(0,s_{0},t^{*})$. This process is clearly primitive recursive.

Recursively suppose that we have computed $s_{y}$ for each $y<x$. To compute $g^{*}(x)$, we apply the following procedure. First, compute $h(a)$ for some $a$ in the $\eta$ part of $A$ to the left of $h^{-1}(b_{\delta_{x-1}\!^{\frown}(s_{x-1}+1)})$. That is, $h(a)$ must be contained in the $\eta$ part of $I_{x-1}$, and therefore have index $>s_{x}$. Repeating a similar procedure as before allows us to obtain $s_{x}$ and some index $t^{*}\geq t_{x,s_{x}}$. Thus, $h\oplus h^{-1}\geq_{PR}g^{*}$. Theorem~\ref{thm:ompeta} is proved.
\end{proof}

\section{Boolean Algebras}\label{sec:ba}

Recall that a Boolean algebra is a (functional) structure closed under the binary operations join (denoted $\vee$), meet (denoted $\wedge$), and the unary operation negation (denoted $\neg$), together with two `special' elements often referred to as the top and the bottom. For the purposes of this article, each Boolean algebra we consider will be of the form $\interval(L)$ for some linear order $L$ defined as follows. The Boolean algebra $\interval(L)$ is the smallest collection of subsets of $L$, consisting of $L$, the intervals $[x,y)$ for each $x,y\in L$, and closed under (finite) union, intersection, and complement. Evidently, $L$ and $\emptyset$ are respectively the top and bottom element of $\interval(L)$, and the operations union, intersection, and complement correspond with the join, meet, and negation operations respectively.

\begin{remark}
As one can see, the $\subseteq$ relation induces a natural partial order on the elements of $\interval(L)$. Using this ordering, we say that an element $a$ is \emph{below} $b$ to mean that $a\vee b=b$ (equivalently $a\wedge b=a$). If we further have that $a\neq b$, then we say that $a$ is \emph{strictly below} $b$.
\end{remark}

When working with Boolean algebras, we shall arrange their enumerations using the tree of generators (see \cite{goncharov97} for the details). Briefly speaking, we only `control' the enumerations of some generating set of the Boolean algebra. Since these are meant to be punctual structures, we implicitly assume that at each stage, new elements for the respective joins and meets are added into our structure as required.

\begin{theorem}\label{thm:d1ba}
    If $B$ is an infinite, relatively computably categorical Boolean algebra, then  we have $\prcat(B) = \cone(\Delta_{1}^{0})$.
\end{theorem}

\begin{proof}
Computably categorical Boolean algebras are finite sums of finite Boolean algebras with the countable atomless Boolean algebra \cite{gd80,remmel81b}. (In addition, the notions of $\Delta^0_1$-categoricity and relative $\Delta^0_1$-categoricity for computable Boolean algebras coincide.) When constructing our punctual copies $A,B$, we enumerate all the finite parts at the very first stage of the construction, and may thus assume that we need only construct $A,B\cong\interval(\eta)$. As usual, given a total computable function $g$, we encode the stage at which $g(x)\downarrow$ for each $x$.

Let $A$ be the standard enumeration of $\interval(\eta)$. In $B$, we reserve special elements $b_{0},b_{1},\dots,$ and leave them temporarily as atoms until $g\downarrow$. More specifically, while waiting for $g(x)\downarrow$, $b_{x}$ will remain as an atom. Once $g(x)\downarrow$, then we begin splitting $b_{x}$. This ensures that any element strictly below $b_{x}$ has index larger than the stage at which the computation $g(x)$ halts.

\

\emph{Construction:} At stage $0$, split the top element of $B$ into $b_{0}$ and $c_{0}\coloneqq\neg b_{0}$. During subsequent stages $s$, split $c_{s-1}$ into $b_{s}$ and $c_{s}\coloneqq c_{s-1}\wedge\neg b_{s}$. In addition, if $g(x)[s]\downarrow$, then begin enumerating a standard copy of $\interval(\eta)$ below $b_{x}$.

\

\emph{Computing $g(x)$:} Let $h:A\to B$ be an isomorphism. First, compute $h^{-1}(b_{x})$. Since $A$ is standard, we may primitive recursively obtain $a,a'$ such that $a\vee a'=h^{-1}(b_{x})$ and $a\wedge a' = \emptyset$. It follows that $h(a)$ should have index larger than the stage at which $g(x)\downarrow$, as $h(a)\vee h(a')=b_{x}$, and such elements are only enumerated into $B$ after $g(x)\downarrow$.
\end{proof}

\

\begin{theorem}\label{thm:d2ba}
    If $B$ is a relatively $\Delta_{2}^{0}$-categorical Boolean algebra which is not relatively $\Delta^0_1$-categorical, then $\prcat(B)=\cone(\Delta_{2}^{0})$.
\end{theorem}

\begin{proof}
It is known that relatively $\Delta_{2}^{0}$-categorical Boolean algebras are finite sums of computably categorical Boolean algebras, with finitely many copies of $\interval(\omega)$ \cite{mccoy03}. In addition, $\Delta^0_2$-categoricity and relative $\Delta^0_2$-categoricity for Boolean algebras coincide \cite{Baz-14}.) Given any Boolean algebra that is relatively $\Delta_{2}^{0}$-categorical, we may enumerate all except a single copy of $\interval(\omega)$ in the standard way. Thus, we simply focus only on constructing $A,B\cong\interval(\omega)$. Let $g^{*}$ be a total $\Delta_{2}^{0}$ function with primitive recursive approximation $g(x,s)$. Just as in the proof of Theorem \ref{thm:d1ba}, we only describe how the generators of $A,B$ will be enumerated. As usual, we use $s_{x}$ to denote the value such that $g(x,s_{x})=g(x,s)$ for all $s\geq s_{x}$.

\

\emph{Encoding $s_{x}$:} $A$ will be enumerated in the `standard' way; $a_{0}$ is the top element, and for each $i\in\omega$, $a_{i}$ is exactly the join of an atom with $a_{i+1}$. In $B$, we reserve special elements $b_{0},b_{1},\dots,$ where the intention is to encode $s_{x}$ with each $b_{x}$. Each $b_{x}$ will split into finitely many atoms and another copy of $\interval(\omega)$. Whenever $g(x,s)\neq g(x,s-1)$, we split one of the atoms in $\neg b_{x}$ into $s$ many atoms. Observe that this process must terminate as $g^{*}$ is assumed to be total, and thus, $\neg b_{x}$ must be made up of only finitely many atoms. Evidently, the number of atoms that comprises $\neg b_{x}$ encodes $s_{x}$.

\

\emph{Construction:} At stage $0$, enumerate $a_{0}$ and $b$ which is the join of $b_{0}$ and a single atom as the top elements of $A$ and $B$ respectively. At each subsequent stage $s$, split $a_{s-1}$ into $a_{s}$ and a single atom. Do the same for $b_{s-1}$. For convenience, we also define $c_{0}=\neg b_{0}$ and $c_{x}=b_{x-1}\wedge\neg b_{x}$. If it is also discovered that $g(x,s)\neq g(x,s-1)$, then split the atoms currently comprising $c_{x}$ until $c_{x}$ is the join of $s$ many atoms.

\

\emph{Computing $g^{*}(x)$:} It is easy to see that $B\cong\interval(\omega)$, and that for each $x\in\omega$, $c_{x}$ is exactly the join of $s_{x}$ many atoms. From the construction, it is evident that $\neg b_{x}$ should be the join of $s_{0}+s_{1}+\dots+s_{x}$ many atoms. For a given isomorphism $h:A\to B$, computing $h^{-1}(\neg b_{x})$ should thus allow us to retrieve the number of atoms that comprises $\neg b_{x}$. In particular, since $A$ is standard, we may retrieve the number of atoms that comprises $h^{-1}(\neg b_{x})$ primitive recursively. Let $i$ be this number. We may thus obtain $g^{*}(x)$ by computing $g(x,i)$.
\end{proof}

\

\begin{theorem}\label{thm:d3ba}
    If $B$ is a relatively $\Delta_{3}^{0}$-categorical Boolean algebra which is not relatively $\Delta^0_2$-categorical, then $\prcat(B)=\cone(\Delta_{3}^{0})$.
\end{theorem}

McCoy~\cite[Theorem~3.14]{mccoy02} proved that relatively $\Delta_{3}^{0}$-categorical Boolean algebras are finite sums of relatively $\Delta^0_2$-categorical Boolean algebras, with finitely many copies of $\interval(\omega*\eta)$ and $\interval(\omega+\eta)$. Therefore, similarly to the proof of Theorem~\ref{thm:d2ba}, here it is sufficient to focus on two cases: $B = \interval(\omega*\eta)$ (see Lemma~\ref{lem:intomteta} below) and $B = \interval(\omega+\eta)$ (Lemma~\ref{lem:intompeta}).

\begin{notation}\label{notat:hat}
    Let $b$ be an element of a Boolean algebra. We use $\hat{b}$ to denote the Boolean algebra below $b$.
\end{notation}

\begin{lemma}\label{lem:intomteta}
    $\prcat(\interval(\omega*\eta))=\cone(\Delta_{3}^{0})$.
\end{lemma}

\begin{proof}
Let $g^{*}$ be a total $\Delta_{3}^{0}$ function, and let $s_{x}^{*}$ and $t_{x,s}^{*}$ be such that for all $s\geq s_{x}^{*},\,\lim_{t}g(x,s,t)=\lim_{t}g(x,s_{x}^{*},t)$, and $\lim_{t}g(x,s,t)=g(x,s,t_{x,s}^{*})$, where $g$ is a primitive recursive approximation of $g^{*}$. We construct punctual structures $A,B\cong\interval(\omega*\eta)$ such that any isomorphism $h:A\to B$ is such that $h\oplus h^{-1}\geq_{PR}g^{*}$. As usual, it suffices to encode for each $x,s$ the values $s_{x}^{*}$ and $t_{x,s}^{*}$.

\

\emph{Encoding $t_{x,s}^{*}$:} We arrange the enumeration of $A$ and $B$ using the usual tree of generators. The idea is to keep the indices of the atoms in $B$ `large', while maintaining a `standard' list of atoms in $A$. By computing the given isomorphism on the `standard' atoms in $A$, we will then be able to retrieve these indices larger than $t_{x,s}^{*}$ by computing the given isomorphism on sufficiently many `standard' atoms in $A$.

For each $n\in\omega$, let the elements $a_{2n+1}$ denote the `standard' atoms. In particular, these elements once enumerated will remain as atoms throughout the construction. These may not be the only atoms in $A$, but nevertheless, they provide an infinite list of atoms which we may use to attempt to retrieve the value $t_{x,s}^{*}$. In $B$, we maintain a list of all elements which are temporarily intended to become atoms. Whenever $g(x,s,t)\neq g(x,s,t-1)$, split the $n^{th}$ element in the list for each $n\geq\langle x,s\rangle$. This ensures that within the list, only $\langle x,s\rangle$ many elements have indices $<t$. The value $t_{x,s}^{*}$ may thus be recovered by computing the given isomorphism on $\langle x,s\rangle+1$ many `standard' atoms in $A$. Since any isomorphism must map atoms to atoms, at least one of the images necessarily has index $>t_{x,s}^{*}$.

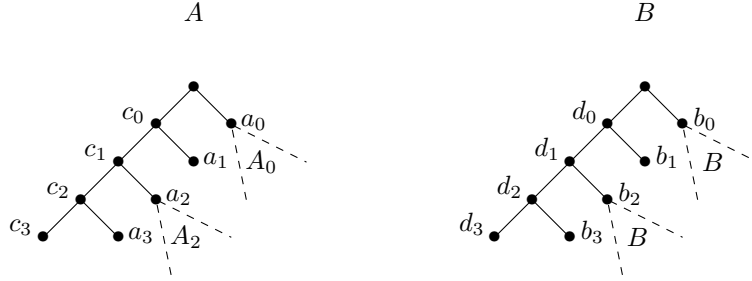
\begin{figure}
    \centering
    \begin{tikzpicture}
    \node at (0,1) {$A$};
    \node at (0,0) {$\bullet$};

    \node at (6,1) {$B$};
    \node at (6,0) {$\bullet$};
    
    \foreach \x/\y in {-0.5/0, -1/1, -1.5/2, -2/3}
    {
    \node at (\x,\x) [label={[shift={(-0.3,-0.3)}]$c_{\y}$}] {$\bullet$};
    \node at (\x+1,\x) [label={[shift={(0.3,-0.4)}]$a_{\y}$}] {$\bullet$};
    \draw (\x,\x) -- (\x+0.5,\x+0.5) -- (\x+1,\x);

    \node at (\x+6,\x) [label={[shift={(-0.3,-0.3)}]$d_{\y}$}] {$\bullet$};
    \node at (\x+7,\x) [label={[shift={(0.3,-0.4)}]$b_{\y}$}] {$\bullet$};
    \draw (\x+6,\x) -- (\x+6.5,\x+0.5) -- (\x+7,\x);
    }

    \foreach \x/\y/\z in {-1/6/B, -1/0/A_{0}, -2/6/B, -2/0/A_{2}}
    {
    \draw[dashed] (\x+\y+1.7,\x-0.5) -- (\x+\y+1.5,\x+0.5) -- (\x+\y+2.5,\x);
    \node at (\x+\y+1.9,\x) {$\z$};
    }
    \end{tikzpicture}
    \caption{The tree of generators for $A\cong B\cong\interval(\omega*\eta)$.}
    \label{fig:intomteta}
\end{figure}

\

\emph{Encoding $s_{x}^{*}$:} Let $P(x,s)$ be the $\Pi_{2}^{0}$ predicate: ``for all $s'\geq s,\,\lim_{t}g(x,s,t)=\lim_{t}g(x,s',t)$''. Referring to Fig.~\ref{fig:intomteta}, within $A$, the elements $a_{2n+1}$ for each $n\in\omega$ will remain as atoms throughout the construction. These will be the `standard' atoms of $A$ mentioned previously. We also construct $A_{2n}$ for each $n\in\omega$ to be isomorphic to $A$ iff $P(x,s)$ where $\langle x,s\rangle=n$ \emph{fires} (defined as usual) infinitely often. As suggested in Fig.~\ref{fig:intomteta}, within $B$, under each element labelled $b_{2n}$ for each $n\in\omega$ is a copy of $B$ itself. The intention here is that these should have preimages $a_{2m}$ (for a given isomorphism from $A$ to $B$), where $\hat{a}_{2m}\cong A$, thus allowing us to obtain the $s$ for some $x$ where $P(x,s)$ fires infinitely often.

To elaborate, for a given isomorphism $h:A\to B$, compute $h^{-1}(b_{0})$. Referring to Fig.~\ref{fig:intomteta}, this should map to an element $a\in A$ such that $\hat{a}$ is isomorphic to $\interval(\omega*\eta)$. Let $n\in\omega$ be the least such that $\hat{a}_{2n}\cong\interval(\omega*\eta)$. By the strategy described above, such an $n$ should be at least as large as $s_{0}$. Since $h$ is an isomorphism,  $a$ should consist of generators with `depth' at least $2n$, as $c_{2n-1}$ is the first generator in $A$ which splits into two disjoint copies of $\interval(\omega*\eta)$. (Observe that $\neg b_{0}$ contains a copy of $\interval(\omega*\eta)$.) Therefore, the index of $h^{-1}(b_{0})$ should be at least as large as $2s_{0}$.

More generally, to recover $s_{x}^{*}$, we compute $h^{-1}$ on sufficiently many (to be specified later) disjoint copies of $\interval(\omega*\eta)$ in $B$. The core idea is to ensure that at least one of these images contains a generator at a sufficiently large `depth' in its representation. However, since we have that $\hat{a}_{2s_{0}}\cong\interval(\omega*\eta)$, the danger is that these $h^{-1}$ images could be spread across those elements generated by the generators below $a_{2s_{0}}$. Thus, some care must also be taken in how exactly we define those generators below the various $a_{2n}$. Let $n=\langle x,s\rangle$ and suppose that $P(x,s)$ fires infinitely often. Below $a_{2n}$, we construct a Boolean algebra isomorphic to $\interval(\omega*\eta)$ similarly to how $A$ is constructed. We define the tree of generators exactly as presented in Fig.~\ref{fig:intomteta}, but with $a_{2n}$ as its top element. Denote the versions of $c_{i}$ and $a_{i}$ below $a_{2n}$ as $c_{i}^{n}$ and $a_{i}^{n}$ respectively. For each $j\leq n$, we leave $a_{2j}^{n}$ as atoms. In particular, the first place at which some generator below $a_{2n}$ splits into two disjoint copies of $\interval(\omega*\eta)$ should have depth at least $s_{x+1}$, if $n=\langle x,s\rangle$ and $P(x,s)$ fires infinitely often. This process shall be repeated recursively in each `nested' copy of $\interval(\omega*\eta)$. To aid in this `nesting' we introduce a system of labels below.

\

\emph{Labelling the generators:} To aid in the description of our tree of generators for $A$, we introduce a labelling system for the generators. Each generator will be labelled via a string in $\omega^{<\omega}$. For each $n\in\omega$, define the generators $c_{n}$ and $a_{n}$ arranged as displayed in Fig.~\ref{fig:intomteta}. Suppose that $a_{\sigma^{\frown}2n}$ has been defined. Then we define $c_{\sigma^{\frown} 2n^{\frown}m}$ and $a_{\sigma^{\frown}2n^{\frown}m}$ as generators below $a_{\sigma^{\frown}2n}$ with the following properties.
\begin{itemize}
    \item Arrange the generators $c_{\sigma^{\frown}2n^{\frown}m}$ and $a_{\sigma^{\frown}2n^{\frown}m}$ as in  Fig.~\ref{fig:intomteta}, replacing each $c_{i}$ with $c_{\sigma^{\frown}2n^{\frown}i}$ and each $a_{i}$ with $a_{\sigma^{\frown}2n^{\frown}i}$, maintaining the respective relations between them. Also define the top element to be $a_{\sigma^{\frown}2n}$.

    \item $\hat{a}_{\sigma^{\frown}2n^{\frown}2m}\cong\interval(\omega*\eta)$ only if all of the following hold: $n=\langle x,s\rangle$, $m=\langle y,t\rangle$, $P(y,t)$ fires infinitely often, $y>x$, and $\hat{a}_{\sigma^{\frown}2n}\cong\interval(\omega*\eta)$.
\end{itemize}
More specifically, we will ensure that $\hat{a}_{\sigma}\cong\interval(\omega*\eta)$ iff for each $i<|\sigma|$, $\sigma(i)=2\langle x_{i},s_{i}\rangle$ where $P(x_{i},s_{i})$ fires infinitely often and $x_{i+1}>x_{i}$.

Adopting the convention that $s_{0}<s_{1}<s_{2}<\dots$, for each $x$, the labels $\sigma$ such that $\hat{a}_{\sigma}\cong\interval(\omega*\eta)$ and $\sigma(i)<2s_{x}^{*}$ for each $i<|\sigma|$, must have length at most $x$ by the desired property above. Furthermore, by adopting the convention that $s_{i}^{*}$ is the unique value for which $P(i,s_{i}^{*})$ fires infinitely, we also obtain that $\sigma(i)\leq 2s_{x-1}^{*}$ for each $i<x$. Thus, the total number of labels with all entries $<2s_{x}^{*}$ is primitive recursive in $x$ and $s_{x-1}^{*}$ (which can be inductively obtained primitive recursively from $x-1$ and the given isomorphism $h$). Let this number be denoted by $f(x)$. By computing $h^{-1}$ on $f(x)+1$ many disjoint copies of $\interval(\omega*\eta)$ in $B$, we must obtain an element with a generator that has label at least $2s_{x}^{*}$. The indices of such disjoint copies of $\interval(\omega*\eta)$ may be obtained using the various $b_{2i}$.

\

\emph{Preliminaries of the construction:} Let $\sigma$ be a label such that for all $i<|\sigma|,\,\sigma(i)=2\langle x_{i},s_{i}\rangle$ for some $x_{i},s_{i}\in\omega$. We say to \emph{grow} $\sigma$ when we enumerate fresh generators $a_{\sigma^{\frown}n}$ and $c_{\sigma^{\frown}n}$ below $a_{\sigma}$. During the construction, we only grow $\sigma$ when it has received \emph{permission} from each $P(x_{i},s_{i})$, defined as follows. Let $s$ be the last stage thus far during which $\sigma$ grew. At a stage $s'>s$, $\sigma$ receives permission from $P(x_{i},s_{i})$ if it has fired at some stage $s''$ where $s<s''\leq s'$. This arrangement ensures that the tree of generators below $a_{\sigma}$ is infinite iff $P(x_{i},s_{i})$ fires infinitely often for each $i<|\sigma|$ and $\sigma(i)=2\langle x_{i},s_{i}\rangle$. By definition, the empty string receives permission at every stage, and thus, $a_{s}$ and $c_{s}$ will be enumerated at stage $s$ of the construction, keeping $A$ punctual.

During the construction, we also maintain an (ordered) list $L$ tracking all generators in $B$ currently intended to be an atom. As explained in the strategy for encoding $t_{x,s}^{*}$, whenever $g(x,s,t)\neq g(x,s,t-1)$, we split the $n^{th}$ generator contained in $L$ where $n\geq\langle x,s\rangle$ into a fresh atom and a copy of $B$. It shall be verified later that this process generates countably many atoms, and ensures that $B\cong\interval(\omega*\eta)$.

\

\emph{Construction:} During stage 0, enumerate the generator $a_{\langle\rangle}$, defining it to be the top element of $A$, and also enumerate the generators $a_{0},c_{0}$ into $A$. Similarly, we enumerate the first three generators in $B$; the top element, and the generators $b_{0},d_{0}$, as shown in Fig.~\ref{fig:intomteta}. Define $L$ to be a currently empty list. During stage $n>0$, do the following.
\begin{description}
    \item[Step 1] Within $B$, split the generator $d_{n-1}$ into $d_{n}$ and $b_{n}$. If $n$ is odd, then add $b_{n}$ into the list $L$. Similarly, split the generator $c_{n-1}$ in $A$ into $a_{n}$ and $c_{n}$. This ensures that both $A$ and $B$ remain punctual.

    \item[Step 2] Let $\langle x,s\rangle\leq n$ be the least such that $g(x,s,n)\neq g(x,s,n-1)$. If no such $x,s$ exists, then we simply proceed to Step 3. Otherwise, let $l_{0},l_{1},\dots,l_{N}$ be the current elements in $L$ by order of their index. For each $m\geq\langle x,s\rangle$, split the generator $l_{m}$ into $b$ and $b^{*}$, adding $b^{*}$ into the list $L$, and removing $l_{m}$ from $L$.

    \item[Step 3] For each $b\in B$ and $b\notin L$, extend $\hat{b}$ by placing a copy of the generators of $B_{n-1}$ ($B$ at the previous stage) with $b$ as the top element. For each $b'\in B_{n-1}\cap L$, place the copy $b''$ of $b'$ under $b$ into $L$. Together with Step 2, this ensures that $L$ contains exactly all the atoms in $B$.

    \item[Step 4] Suppose that $a_{\sigma}$ has been enumerated into $A$. Additionally assume that for each $i<|\sigma|,\,\sigma(i)=2\langle x_{i},s_{i}\rangle$, for some $x_{i},s_{i}\in\omega$ where $x_{i}<x_{i+1}$. If $\sigma$ has permission at stage $n$ from each $P(x_{i},s_{i})$, then we grow $\hat{a}_{\sigma}$ by enumerating the generators $a_{\sigma^{\frown}m}$ and $c_{\sigma^{\frown}m}$ for each $m\leq n$ if they have yet to be enumerated. Observe that if $\sigma$ does not satisfy our assumptions, then we never grow $\sigma$; $a_{\sigma}$ remains as an atom.
\end{description}

\

\emph{Verification:} We split the proof into the following steps.
\begin{enumerate}[label=(\roman*)]
    \item First, we show that for each $\sigma\in\omega^{<\omega}$, $\hat{a}_{\sigma}\cong\interval(\omega*\eta)$ iff for each $i<|\sigma|,\,\sigma(i)=2\langle x_{i},s_{i}\rangle$, where $P(x_{i},s_{i})$ fires infinitely often and $x_{i}<x_{i+1}$.

    \item Second, we show that $b\in B$ is an atom iff $b\in L$, and if $b\notin L$, then $\hat{b}\cong\interval(\omega*\eta)$. In addition, the $\langle x,s\rangle$-th element of $L$ (by order of indices) has an index at least as large as $t_{x,s}^{*}$. (Recall that this is the value such that $\lim_{t}g(x,s,t)=g(x,s,t_{x,s}^{*})$.)
\end{enumerate}
Since there is exactly one $s_{x}^{*}\in\omega$ for each $x\in\omega$, such that $P(x,s_{x}^{*})$ fires infinitely often, (i) ensures that $A\cong\interval(\omega*\eta)$. Similarly, (ii) also ensures that $B\cong\interval(\omega*\eta)$.

To prove (i), we proceed by induction on $|\sigma|$ and show that if $\sigma$ grows infinitely often, then there are infinitely many $m$ for which $\sigma^{\frown}m$ grows infinitely often. The base case is trivial; the only label of length $0$ is the empty string, and all labels of the form $ 2\langle x,s_{x}^{*}\rangle$ grows infinitely often, since $P(x,s_{x}^{*})$ must fire infinitely often by choice of $s_{x}^{*}$. Suppose that $\sigma=2\langle x_{0},s_{0}\rangle^{\frown}2\langle x_{1},s_{1}\rangle^{\frown}\dots^{\frown}2\langle x_{k},s_{k}\rangle$. Recall that in order for $\sigma$ to grow, it has to be that $x_{0}<x_{1}<\ldots<x_{k}$ and that $P(x_{i},s_{i})$ all fire infinitely often for each $i\leq k$. That is, labels of the form $\sigma^{\frown}m$ for any $m\in\omega$ receive permission from $P(x_{i},s_{i})$ infinitely often. If $m$ is also such that $m=2\langle x,s\rangle$ where $x>x_{k}$ and $P(x,s)$ fires infinitely often, then $\sigma^{\frown}m$ grows infinitely often. It is easy to see that there are infinitely many such $m$. Furthermore, these are also exactly the labels which grow infinitely often.

To complete the proof of (i), it suffices to note that $\hat{a}_{\sigma}$ is isomorphic to $A$, with the exception of finitely many atoms. But since $A$ has countably many atoms, this in fact provides an isomorphism from $\hat{a}_{\sigma}$ to $A$ \cite{remmel81b}. In other words, each generator of $A$ either splits into two disjoint copies of $A$, or splits into one copy of $A$ and a finite Boolean algebra, both of which happen infinitely often, and thus, $A\cong\interval(\omega*
\eta)$.

A quick analysis of Steps 1, 2, and 3, of the construction allows us to conclude that, $b\in B$ is an atom iff $b\in L$. Furthermore, if $b\notin L$, then by Step 3 of the construction, at each stage $n$, $\hat{b}\cong B_{n-1}$. Thus, we have that $\hat{b}\cong B$ in the limit for each $b\notin L$. It is obvious that there are infinitely many generators of $B$ that are not in $L$. It remains to argue that $|L|=\infty$. By Step 1 of the construction, after every two stages, at least one new generator is added into $L$. Let $x,s\in\omega$ be given. At a stage $n^{*}$ large enough such that $n^{*}>\max\{t_{y,m}^{*}\mid\langle y,m\rangle\leq\langle x,s\rangle\}$, the current first $\langle x,s\rangle$ many elements of $L$ in order of their indices will never be removed from $L$ via Step 2 of the construction after stage $n^{*}$. It follows immediately that $|L|=\infty$, and therefore, $B\cong\interval(\omega*\eta)$.

\

\emph{Computing $g^{*}(x)$:} Finally, using (i) and (ii), given an isomorphism $h:A\to B$ and $x\in\omega$, we recover $g^{*}(x)$ as follows. First compute $h^{-1}(b_{2i})$ for each $i\leq f(x)+1$ (recall that one more than the number of distinct labels that grow infinitely often and have entries all bounded by $s_{x}^{*}$). By property (i) and the argument before, at least one of the $h^{-1}(b_{i})$ must have index $\geq s_{x}^{*}$. Let the largest index obtained from this computation be $s'\geq s_{x}^{*}$. Using this $s'$, compute $h(a_{2i+1})$ for each $i\leq \langle x,s'\rangle$. Since $a_{2i+1}$ must be an atom, then these have to map to atoms of $B$. In particular, they must map to $\langle x,s'\rangle+1$ many distinct elements of $L$. By pigeonhole principle, at least one of the $h$ images of $a_{2i+1}$ should have index $t'\geq t_{x,s'}^{*}$. We thus obtain that $g(x,s',t')=\lim_{t}g(x,s',t)=\lim_{s}\lim_{t}g(x,s,t)=g^{*}(x)$. Lemma~\ref{lem:intomteta} is proved.
\end{proof}

\begin{lemma}\label{lem:intompeta}
    $\prcat(\interval(\omega+\eta))=\cone(\Delta_{3}^{0})$.
\end{lemma}

\begin{proof}
We adopt the usual proof strategy and aim to encode the values $s_{x}^{*}$ and $t_{x,s}^{*}$ for each $x,s\in\omega$, which allows us to compute a total $\Delta_{3}^{0}$ function; $g^{*}(x)=\lim_{s}\lim_{t}g(x,s,t)=\lim_{t}g(x,s_{x}^{*},t)=g(x,s_{x}^{*},t_{x,s_{x}^{*}}^{*})$.

\

\emph{Encoding $s_{x}^{*}$:} As before, let $P(x,s)$ be the predicate: ``for all $s'\geq s,\,\lim_{t}g(x,s',t)=\lim_{t}g(x,s,t)$''. It is evident that $P(x,s)$ is $\Pi_{2}^{0}$, and we say that it \emph{fires} at some stage $n$ if it temporarily looks to be true. We further assume that for each $x$, the unique $s=s_{x}^{*}$ at which $P(x,s_{x}^{*})$ fires infinitely often is such that $P(x,s)$ never fire after stage $s_{x}^{*}$ for any $s<s_{x}^{*}$ (folklore). This property shall be key in the strategy for encoding $t_{x,s}^{*}$ to be discussed later. It is known that the generators of $\interval(\omega+\eta)$ either splits into a copy of itself and a finite Boolean algebra or it splits into a copy of itself and a copy of $\interval(\eta)$. The idea here is to associate the truth values of $P(x,s)$ with a pair of generators. If $P(x,s)$ fires infinitely often, then the generator associated with $P(x,s)$ being true will be isomorphic to $\interval(\omega+\eta)$ and mutatis mutandis if $P(x,s)$ fires only finitely.

We first consider how the various generators for $P(0,s)$ may be arranged. We label the generators via tuples, $\langle s,i\rangle$ for each $i<3$, with the following relations. For each $s\in\omega$, $a_{\langle s,0\rangle}=\bigvee_{i<3}a_{\langle s+1,i\rangle}$, and $a_{\langle s,i\rangle}\wedge a_{\langle s,j\rangle}=\emptyset$ for any $i\neq j$. Also define $\bigvee_{i<3}a_{\langle 0,i\rangle}$ to be the top element of $A$. Whenever $P(0,s)$ fires, we `grow' (to be defined later) the Boolean algebra below $a_{\langle s,1\rangle}$, and for each $s'\geq s$, split every atom currently below $a_{\langle s',2\rangle}$ and `grow' every generator currently below $a_{\langle s',0\rangle}$. Let $s_{0}^{*}$ be the unique value for which $P(0,s_{0}^{*})$ fires infinitely often. Then for each $s\in\omega$, the construction should ensure that all the following properties hold.
\begin{itemize}
    \item If $s<s_{0}^{*}$, then $\hat{a}_{\langle s,0\rangle}\cong\interval(\omega+\eta)$, and both $\hat{a}_{\langle s,2\rangle}$ and $\hat{a}_{\langle s,1\rangle}$ are finite.

    \item If $s>s_{0}^{*}$, then $\hat{a}_{\langle s,0\rangle}\cong\hat{a}_{\langle s,1\rangle}\cong\hat{a}_{\langle s,2\rangle}\cong\interval(\eta)$. 

    \item If $s=s_{0}^{*}$, then $\hat{a}_{\langle s,0\rangle}\cong\hat{a}_{\langle s,2\rangle}\cong\interval(\eta)$ and $\hat{a}_{\langle s,1\rangle}\cong\interval(\omega+\eta)$.
\end{itemize}
In order to retrieve $s_{0}^{*}$, we maintain a primitive recursive list of generators $c$ within $B$ such that $\hat{c}\cong\interval(\omega+\eta)$. For any given isomorphism $h:A\to B$, the idea is for $h^{-1}(c)$ to map to some element within $A$ with index at least as large as $s_{0}^{*}$. However, since isomorphisms of Boolean algebras may `shuffle' the generators, some additional work is required to ensure that the strategy works.

In $B$, we have the following types of generators: $c_{n}$, $c_{n}^{*}$, $b_{\gamma}$ and $d_{\gamma}$ for each $n\in\omega$ and each $\gamma\in\omega^{<\omega}$. We will specify the relations between these generators shortly, but for now, we need only the intuition that $c_{i}$ will have the property that $\hat{c}_{i}\cong\interval(\omega+\eta)$. We return our focus to retrieving some index $\geq s_{0}^{*}$; for a given isomorphism $h:A\to B$, we would like to have the property that $h^{-1}(c_{0})$ contains the generator $a_{\langle s_{0}^{*},1\rangle}$ in its representation. To this end, we shall ensure that $\hat{c}_{0}\cong\interval(\omega+\eta)$ and $\neg\hat{c}_{0}$ contains more atoms than $\neg\hat{a}_{\langle s_{0}^{*},1\rangle}$. We illustrate this in Fig.~\ref{fig:intompeta}. Within $B$, whenever $P(0,s)$ fires, we grow the Boolean algebra below $b_{\gamma_{s'}}$ for each $s'\geq s$. Thus, it is evident that if $P(0,s)$ fires infinitely often ($s=s_{0}^{*}$), then $\hat{d}_{\gamma_{s}}\cong\interval(\eta)$, and $\hat{b}_{\gamma_{s'}}$ is finite for each $s'<s$.\footnote{The general description (see ``Controlling the number of atoms in $B$'') shall differ slightly from the simple setup here, although with similar intuition.} The details of how each $b_{\gamma}$ is grown in general may be found later.
\begin{figure}
    \centering
    \begin{tikzpicture}
        \node at (0,1) {$A$};
        \node at (0,0) {$\bullet$};
        \node at (1,-1) [label={[shift={(0.5,-0.3)}]$a_{\langle s-1,0\rangle}$}] {$\bullet$};
        \draw[-,dashed] (-1.5,-0.5) -- (0,0) -- (0,-1);
        \node at (-0.6,-0.6) {\tiny finite};
        \draw[-,dotted] (0,0) -- (1,-1);
        
        \foreach \x/\y in {1/0, 2.5/2, -0.5/1}
        {
        \node at (\x,-2) [label={[shift={(0.5,-0.5)}]$a_{\langle s,\y\rangle}$}] {$\bullet$};
        \draw[-] (1,-1) -- (\x,-2);
        }
        \draw[-,dashed] (-1.5,-2.5) -- (-0.5,-2) -- (-0.5,-3);
        \draw[-,dashed] (0,-3) -- (1,-2) -- (2,-3);
        \draw[-,dashed] (2.5,-3) -- (2.5,-2) -- (3.5,-2.5);

        \node at (-1.2,-2.7) {\tiny$\interval(\omega+\eta)$};
        \node at (1,-2.8) {\tiny$\interval(\eta)$};
        \node at (3,-2.7) {\tiny$\interval(\eta)$};

        \node at (6,1) {$B$};
        \node at (6,0) {$\bullet$};
        \foreach \x/\y in {0/0.4, 1/0.8}
        {
        \node at (6-\y,-\x-1) [label={[shift={(0.3,-0.3)}]$c_{\x}$}] {$\bullet$};
        }
        \draw[-] (6,0) -- (5.2,-2);
        %\draw[-,dashed] (4.4,-1) -- (5.6,-1) -- (4.9,-1.8);
        
        \draw[-,dashed] (4.2,-3) -- (5.2,-2) -- (6.2,-3);
        \node at (5.2,-2.8) {\tiny$\interval(\omega+\eta)$};
        \node at (5,-1.2) {$\bullet$};
        \draw[-] (5,-1.2) -- (5.6,-1);
        \draw[-,dashed] (4.7,-2) -- (5,-1.2) -- (4,-1.2);
        \node at (4.2,-1.5) {\tiny$\interval(\eta)$};
        \node at (4.2,-1.7) {\tiny$\vee$ finite};
        
        \foreach \x/\y/\a/\d in {1/0.5/0/c_{0}^{*}, 2/1/1/d_{\gamma_{0}}}
        {
        \node at (6+\x,-\y) [label={[shift={(0.4,-0.4)}]$\d$}] {$\bullet$};
        \node at (5.8+\x,-\y-0.5) [label={[shift={(0.4,-0.5)}]$b_{\gamma_{\a}}$}] {$\bullet$};
        \draw[-] (6+\x,-\y) -- (5.8+\x,-\y-0.5);
        \draw[-,dashed] (5+\x,-\y-0.9) -- (5.8+\x,-\y-0.5) -- (5.9+\x,-\y-1.2);
        \node at (5.4+\x,-\y-1) {\tiny finite};
        }
        \draw[-] (6,0) -- (8,-1);
        \draw[-,dotted] (8,-1) -- (9,-1.5);
        \node at (9,-1.5) [label={[shift={(0.5,-0.4)}]$d_{\gamma_{s}}$}] {$\bullet$};
        \draw[-,dashed] (8.5,-2.3) -- (9,-1.5) -- (9.7,-2.3);
        \node at (9.1,-2.1) {\tiny$\interval(\eta)$};
    \end{tikzpicture}
    \caption{$A$ and $B$ if $P(0,s)$ fires infinitely often.}
    \label{fig:intompeta}
\end{figure}
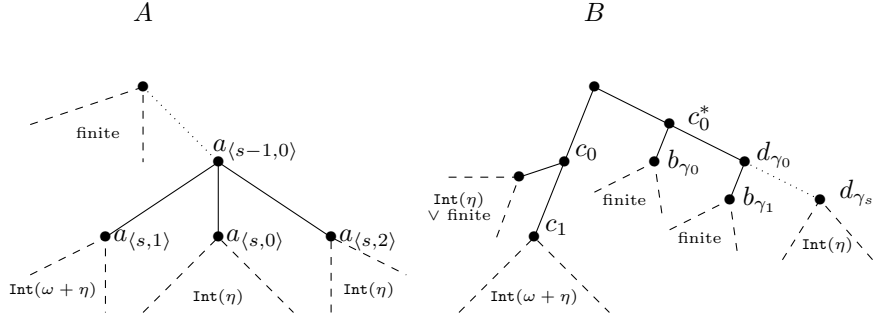
The strategy will be iterated in order to encode the various $s_{x}^{*}$. Before we describe this iterated strategy, we introduce a labelling system for the generators of $A$.

\

\emph{The labels for $A$:} The labels will be finite strings, $\sigma$, where the $x^{th}$ entry is a tuple $\langle s,i\rangle$ with the following properties.
\begin{itemize}
    \item $s\in\omega$ and $i\in\{0,1,2\}$.

    \item If $\sigma$ has final entry $\langle s,0\rangle$, then $a_{\sigma}=\bigvee_{i<3}a_{(\sigma\restriction |\sigma|-1)^{\frown}\langle s+1,i\rangle}$. The intuition here is that if $\sigma$ has final entry $\langle s,0\rangle$, this node is associated with the outcome where $P(|\sigma|,s)$ does not hold. Thus, we extend it by nodes which encode the predicate $P(|\sigma|,s+1)$.

    \item If $\sigma$ has final entry $\langle s,1\rangle$, then $a_{\sigma}=\bigvee_{i<3}a_{\sigma^{\frown}\langle 0,i\rangle}$. The intuition here is similar. If the final entry of $\sigma$ is $\langle s,1\rangle$, then $\sigma$ is associated with the outcome where $P(|\sigma|,s)$ holds, and thus, we extend it by nodes which begin the process to encode $P(|\sigma|+1,0)$.

    \item Observe that the labels are only extended if all its entries are of the form $\langle x,1\rangle$. In addition, if $\sigma$ is a label, then the first entry of $\sigma(x+1)$ is exactly one more than the first entry of $\sigma(x)$. Equivalently, only the final entry of a label $\sigma$ may itself possibly have a second entry that is either $0$ or $2$.
\end{itemize}
Evidently, the idea is that the $x^{th}$ entry of each label $\sigma$ encodes the outcome of $P(\sigma(x))$ which $\sigma$ `believes' in. More specifically, $\hat{a}_{\sigma}\cong\interval(\omega+\eta)$ iff for each $x<|\sigma|-1$, $\sigma(x)=\langle s_{x}^{*},1\rangle$, and $\sigma(|\sigma|-1)=\langle s,i\rangle$ for some $s<s_{x}^{*}$ and $i=0$, or $s=s_{x}^{*}$ and $i=1$ (this shall be formally verified later).

At the beginning of the construction, enumerate the top element $a_{\langle\rangle}$ into $A$. At each stage $n$ of the construction, $a_{\langle\rangle}$ will be given the instruction to \emph{grow}; enumerate the generators $a_{\langle n,i\rangle}$ for each $i<3$ (recall that $\bigvee_{i<3}a_{\langle 0,i\rangle}=a_{\langle\rangle}$). More generally, a generator labelled with $\sigma=\langle s_{0},i_{0}\rangle^{\frown}\langle s_{1},i_{1}\rangle^{\frown}\dots^{\frown}\langle s_{x},i_{x}\rangle$ grows at a stage $n$ iff the following hold. Suppose that $m$ is the most recent stage during which $a_{\sigma}$ grows. If after stage $m$, there are stages during which $a_{\sigma\restriction |\sigma|-1}$ grows, and $P(x,s_{x})$ fires (not necessarily at the same stage), then we grow $a_{\sigma}$. Depending on the final entry of $\sigma$, we grow the Boolean algebra below $a_{\sigma}$ differently.
\begin{itemize}
    \item If $i_{x}=0$, then enumerate all the generators $a_{\tau}$ below $a_{\sigma}$ up to a depth of $n$. Also ensure that for each $\tau$ where the final entry of $\tau$ itself has final entry $2$ is currently isomorphic to the standard presentation of $\interval(\eta)$ at stage $n$. In other words, if $a_{\sigma}$ grows infinitely often, then $\hat{a}_{\sigma}\cong\interval(\eta)$ and is finite otherwise.

    \item If $i_{x}=1$, then enumerate the generators $a_{\sigma^{\frown}\langle s,i\rangle}$ for each $s\leq n$ and each $i<3$. Intuitively, we only enumerate a `spine' below $a_{\sigma}$. Obviously, if $a_{\sigma}$ grows only finitely, then $\hat{a}_{\sigma}$ is finite. We claim that if $a_{\sigma}$ grows infinitely, then $\hat{a}_{\sigma}\cong\interval(\omega+\eta)$.

    \item If $i_{x}=2$, then split every atom currently below $a_{\sigma}$. Then $\hat{a}_{\sigma}\cong\interval(\eta)$ if $a_{\sigma}$ grows infinitely, and is finite otherwise. 
\end{itemize}
With this system of labels in place, given an isomorphism $h:A\to B$, the strategy to retrive $s_{x}^{*}$ is to first compute $h^{-1}(c_{i})$ for each $i\leq x$. The idea here is that each of these should map to some element within $A$ which is the join of generators, at least one of which has index $\geq s_{i}^{*}$. Recall that the way to ensure this holds is by keeping the number of atoms in $\neg c_{i}$ larger than the number of atoms in $\neg a_{\sigma}$ for any `small' $\sigma$.

\

\emph{Controlling the number of atoms in $B$:} For each $x\in\omega$, let $\gamma_{x,0},\gamma_{x,1},\dots$ denote the labels of length $x+1$ in the order in which the generators $a_{\gamma_{x,i}}$ are enumerated into $A$. (Note that if $a_{\sigma}$ is never enumerated into $A$, then $\sigma$ will not be listed by this procedure.) In $B$, the top element will be given by $d_{\gamma_{0,0}}\vee c_{0}$. For each $x,n\in\omega$, we have the following.
\begin{itemize}
    \item $c_{x}=c_{x+1}\vee c_{x+1}^{*}$.

    \item $c_{x}^{*}=d_{\gamma_{x,0}}\vee b_{\gamma_{x,0}}$.
    
    \item $d_{\gamma_{x,n}}=d_{\gamma_{x,n+1}}\vee b_{\gamma_{x,n+1}}$.
\end{itemize}
The intention here is to keep the number of atoms below $b_{\gamma_{x,n}}$ larger than the number of atoms below $a_{\gamma_{x,n}}$. By always splitting all current atoms below $b_{\gamma_{x,n}}$ whenever the number of atoms below $a_{\gamma_{x,n}}$ grow, we may ensure that $\hat{b}_{\gamma_{x,n}}$ is either finite and larger than $\hat{a}_{\gamma_{x,n}}$, or $\interval(\eta)$. In fact, we ensure that $\hat{b}_{\gamma_{x,n}}$ is finite iff $\hat{a}_{\gamma_{x,n}}$ is also finite. We shall verify later that for each $x$, there must be cofinitely many generators $b_{\gamma_{x,n}}$ such that the Boolean algebra below it is isomorphic to $\interval(\eta)$. It is also evident that $\neg\hat{c}_{x}$ will thus have more atoms than $\neg\hat{a}_{\sigma}$ for any $\sigma$ not containing any index $\geq s_{x}^{*}$. By computing $h^{-1}(c_{i})$ up to $i=x+1$, we may obtain an index at least as large as $s_{x}^{*}$.

\

\emph{Encoding $t_{x,s}^{*}$:} To give a summary of the strategy thus far, within $B$, we have `standard' copies of $\interval(\omega+\eta)$ given by $\hat{c}_{i}$, and we need to be able to control the number of atoms below $c_{\gamma_{i,0}}$ for each $i\in\omega$. Within $A$, we have `non-standard' copies of $\interval(\omega+\eta)$ with `large' indices to encode $s_{x}^{*}$. Unlike the strategy for the proof of Lemma \ref{lem:intomteta}, we do not seem to be able to place `standard' atoms in either of our structures $A$ or $B$ to encode $t_{x,s}^{*}$. To circumvent this issue, we shall adopt a strategy similar to the one utilised in the proof of Theorem \ref{thm:ompeta}.

Recall the assumption that $s_{x}^{*}$ is large enough such that for each $s<s_{x}^{*}$, $P(x,s)$ never again fires after stage $s_{x}^{*}$. In particular, the atoms below the various $a_{\langle s,2\rangle}$ for each $s<s_{0}^{*}$ never again changes after stage $s_{0}^{*}$, as $a_{\langle s,2\rangle}$ never again grows after stage $s_{0}^{*}$. Taking the convention that $P(0,s)$ each fires at least once, there should be at least $2^{s_{0}^{*}-1}$ many atoms below $a_{\langle s_{0}^{*}-1,2\rangle}$ by stage $s_{0}^{*}$. We may assume that this is larger than $\langle 0,s_{0}^{*}\rangle$, as there exists pairing functions which are polynomial in the inputs. By computing the given isomorphism on these $2^{s_{0}^{*}-1}$ many atoms, the hope is to retrieve a value $\geq t_{0,s_{0}^{*}}^{*}$. Similar ideas can be applied to obtain the indices of a `large' number (relative to $\langle x,s_{x}^{*}\rangle$) of atoms primitive recursively in $x$ and the given isomorphism.

The technique used in the proof of the previous lemma may now be applied. We maintain a list $L$ of the atoms in $B$ at each stage. Whenever, the approximation $g(x,s,t)\neq g(x,s,t-1)$, we split the $i^{th}$ atom in $L$ (ordered by their indices) for each $i\geq\langle x,s\rangle$, thereby ensuring that there are $<\langle x,s\rangle$ many atoms in $B$ with index less than the current stage number. As the number of atoms in $B$ below each generator $b_{\sigma}$ needs to be carefully controlled, when we split atoms in $B$ for the sake of encoding $t_{x,s}^{*}$, we do so without changing the total number of atoms currently in $B$. This can be achieved by splitting an atom into a fresh atom with a copy of $\interval(\eta)$ (note that this will not change the isomorphism type of $B$).

\

\emph{Construction:} At stage $n$ of the construction, do the following.
\begin{description}
    \item[Step 1] Enumerate the generator $a_{\langle n,i\rangle}$ for each $i<3$ into $A$ ($a_{\langle\rangle}$ grows at every stage). Similarly, enumerate the generators $c_{i}^{*}$ and $c_{i}$ into $B$ for each $i\leq n$. This ensures that both $A$ and $B$ are punctual.

    \item[Step 2] For each $x,s\in\omega$ such that $\langle x,s\rangle\leq n$, check if $P(x,s)$ fires. Let $m<n$ be the most recent stage during which $a_{\sigma}$ grows, where $\sigma(x)=\langle s,i\rangle$ for some $i<3$ (these generators all grow at the same stages). If there is some stage between $m$ and $n$ during which $a_{\sigma\restriction |\sigma|-1}$ grows, then grow $a_{\sigma}$. Otherwise do not grow $a_{\sigma}$ at this stage. This ensures that if $a_{\tau}$ grows infinitely often, then $a_{\gamma}$ for any prefix $\gamma$ of $\tau$ must also grow infinitely often.
    
    \item[Step 3] Within $B$, for each $x\leq n$, and each label $\sigma$ of length $x+1$ such that $a_{\sigma}$ is currently in $A$, enumerate the generator $b_{\sigma}$ into $B$. Also ensure that for each $b_{\sigma}$ currently in $B$, there are more atoms below $b_{\sigma}$ than the total number of atoms currently below $a_{\sigma}$. Whenever we grow the number of atoms below $b_{\sigma}$, always split the atom with the current least index below $b_{\sigma}$. In this way, if the number of atoms below $a_{\sigma}$ tends to infinity, $\hat{b}_{\sigma}\cong\interval(\eta)$ in the limit.

    \item[Step 4] Let $L=\{d_{0},d_{1},\dots,d_{N}\}$ be the list of all generators currently in $B$, arranged in order of their index, which are currently intended to be atoms. For each $x,s\in\omega$, if $g(x,s,n)\neq g(x,s,n-1)$, then split $d_{i}$ for each $i\geq\langle x,s\rangle$ into an atom with a fresh index and a copy of $\interval(\eta)$. Split the atoms in order of their index; the previous $\langle x,s\rangle$-th atom splits into the new $\langle x,s\rangle$-th atom and a copy of $\interval(\eta)$.
\end{description}

\

\emph{Verification:} First we show that for each $x$, the label $\sigma_{x}=\langle s_{0}^{*},1\rangle^{\frown}\langle s_{1}^{*},1\rangle^{\frown}\dots^{\frown}\langle s_{x-1}^{*},1\rangle$ is the unique label such that all of the following hold.
\begin{itemize}
    \item For any $i<3$, $a_{\sigma_{x}\!^{\frown}\langle s,i\rangle}$ grows infinitely often iff $s=s_{x}^{*}$.
    \item For each $s<s_{x}^{*}$ and each $i\in\{1,2\}$, $\hat{a}_{\sigma_{x}\!^{\frown}\langle s,i\rangle}$ is finite.
    \item For each $s>s_{x}^{*}$ and each $i<3$, $\hat{a}_{\sigma_{x}\!^{\frown}\langle s,i\rangle}\cong\interval(\eta)$.
\end{itemize}
For the base case $x=0$, since $s_{0}^{*}$ is the unique value for which $P(0,s_{0}^{*})$ fires infinitely often, and the generators $a_{\langle s,i\rangle}$ for each $s,i$ grows only when $P(0,s)$ fires, thus $a_{\langle s,i\rangle}$ grows infinitely often iff $s=s_{0}^{*}$. For each $s<s_{0}^{*}$, since $a_{\langle s,i\rangle}$ grows only finitely, by Step 2 of the construction, any generator below $a_{\langle s,1\rangle}$ also grows only finitely often, and therefore, $\hat{a}_{\langle s,1\rangle}$ is finite. Similarly, as new generators are enumerated below $a_{\langle s,2\rangle}$ only if $P(0,s')$ fires for some $s'\leq s$, then we may also conclude that $\hat{a}_{\langle s,2\rangle}$ is finite for each $s<s_{0}^{*}$. Finally, the generators $a_{\langle s,i\rangle}$ for any $s>s_{0}^{*}$ and any $i<3$ are all below $a_{\langle s_{0}^{*},0\rangle}$, which grows infinitely often, and thus, $\hat{a}_{\langle s,i\rangle}\cong\interval(\eta)$. The uniqueness of $\sigma_{0}=\langle\rangle$ is trivial as its the only string of length $0$.

Inductively suppose that the statement is true for all $y<x$. By the inductive hypothesis, $a_{\sigma_{x}}$ grows infinitely often, and thus, whether or not a generator below $a_{{\sigma_{x}}\!^{\frown}\langle s,i\rangle}$ grows depends only on the predicate $P(x,s)$. Just as before, since $s_{x}^{*}$ is the unique value for which $P(x,s_{x}^{*})$ fires infinitely, we again obtain that for each $i<3$, $a_{\sigma_{x}\!^{\frown}\langle s,i\rangle}$ grows infinitely often iff $s=s_{x}^{*}$. For the same reasons as before, $\hat{a}_{\sigma_{x}\!^{\frown}\langle s,1\rangle}$ and $\hat{a}_{\sigma_{x}\!^{\frown}\langle s,2\rangle}$ are finite for each $s<s_{x}^{*}$, and for each $s>s_{x}^{*}$, $\hat{a}_{\sigma_{x}\!^{\frown}\langle s,i\rangle}\cong\interval(\eta)$ for any $i<3$. It remains to check that $\sigma_{x}$ is unique.

Suppose that $\sigma\neq\sigma_{x}$ also has the desired properties. Let $y<x$ be the least such that $\sigma_{x}(y)\neq\sigma(y)$. In order for $\sigma^{\frown}\langle s_{x}^{*},1\rangle$ to be a label, $|\sigma|=x$ and $\sigma(y)=\langle s_{y},1\rangle$ for each $y<x$ where $s_{y}\in\omega$. By the definition of $\sigma_{x}$, we obtain that $s_{y}\neq s_{y}^{*}$, otherwise, $\sigma(y)=\sigma_{x}(y)$. In addition, $\sigma\restriction y+1=(\sigma_{x}\restriction y)^{\frown}\langle s_{y},1\rangle=\sigma_{y}\!^{\frown}\langle s_{y},1\rangle$. Applying the inductive hypothesis, $a_{\sigma\restriction(y+1)}=a_{\sigma_{y}\!^{\frown}\langle s_{y},1\rangle}$ grows only finitely often. This further implies that the generators $a_{\tau}$ for any $\tau$ extending $\sigma\restriction y+1$ also grows only finitely often (see Step 2 of the construction). Thus, it cannot be that $a_{\sigma^{\frown}\langle s_{x}^{*},1\rangle}$ grows infinitely often. Therefore, $\sigma_{x}$ must be unique.

Applying the above allows us to obtain that for each $x\in\omega$, and each $s<s_{x}^{*}$, the generator $a_{\sigma_{x}^{\frown}\langle s,0\rangle}$ splits into
\begin{itemize}
    \item $a_{\sigma_{x}\!^{\frown}\langle s+1,1\rangle}\vee a_{\sigma_{x}\!^{\frown}\langle s+1,2\rangle}$ which is finite, and $a_{\sigma_{x}\!^{\frown}\langle s+1,0\rangle}$ if $s\neq s_{x}^{*}-1$, or
    \item $a_{\sigma_{x}\!^{\frown}\langle s_{x}^{*},0\rangle}\cong a_{\sigma_{x}\!^{\frown}\langle s_{x}^{*},2\rangle}$ which is isomorphic to $\interval(\eta)$, and $a_{\sigma_{x}\!^{\frown}\langle s_{x}^{*},1\rangle}$ if $s=s_{x}^{*}-1$.
\end{itemize}
Thus, $A\cong\interval(\omega+\eta)$, where $\hat{a}_{\sigma}\cong\interval(\omega+\eta)$ iff $\sigma$ is such that the $x^{th}$ entry is $\langle s_{x}^{*},1\rangle$ except possibly the final entry which may be $\langle s,0\rangle$ for some $s<s_{|\sigma|-1}^{*}$.

Now we claim that for each $x\in\omega$, and for each label $\sigma$ of length $x+1$, $\hat{b}_{\sigma}$ is either $\interval(\eta)$ or is finite and contains more atoms than $\hat{a}_{\sigma}$ (recall Notation \ref{notat:hat}). Evidently, the atoms in $B$ are affected only by Steps 3 or 4 of the construction. Let $\sigma$ of length $x+1$ be given, and consider the Boolean algebra $\hat{b}_{\sigma}$. If $\hat{a}_{\sigma}\cong\interval(\omega+\eta)$ or $\hat{a}_{\sigma}\cong\interval(\eta)$, then by Step 3 of the construction, $\hat{b}_{\sigma}\cong\interval(\eta)$ as every atom below $b_{\sigma}$ is eventually split. We may thus suppose that $\hat{a}_{\sigma}$ is finite as it is the only remaining possibility for the isomorphism type of $\hat{a}_{\sigma}$. After the stage at which the number of atoms below $\hat{a}_{\sigma}$ stabilises, the number of atoms below $\hat{b}_{\sigma}$ can no longer change due to Step 3. This is however not sufficient to show that there are atoms below $b_{\sigma}$ at all in the limit; if newly introduced atoms are always split, then $\hat{b}_{\sigma}\cong\interval(\eta)$. We now show also that Step 4 of the construction only splits the atoms below $b_{\sigma}$ finitely many times, hence ensuring that $\hat{b}_{\sigma}$ contains more atoms than $\hat{a}_{\sigma}$ in the limit.

In Step 4 of the construction, an atom with index $i$ in $B$ is split at stage $n$ only when $g(x,s,n)\neq g(x,s,n-1)$ for some $x,s$ where $\langle x,s\rangle\leq i$. Let $L[n]$ denote the list $L$ during stage $n$. Also let $l_{i}[n]$ denote the $i^{th}$ element of $L[n]$. Since $g(x,s,n)$ eventually stabilises, we also obtain that $l_{i}[n]$ eventually stabilises. It remains to prove that the position of these atoms also eventually stabilise. More specifically, let $e_{0}[n],e_{1}[n],\dots,e_{m}[n]$ denote the atoms currently below $b_{\sigma}$ at stage $n$ after the number of atoms below it has stabilised. We show that for each $i$, there exists $j$ and $n^{*}$ such that for all $n'\geq n^{*},\,e_{i}[n']=l_{j}[n']$. Applying this with the fact that each $l_{j}$ stabilises, we obtain that the indices of atoms below $b_{\sigma}$ also stabilises. Let $i$ be given. Evidently, $e_{i}[n]=l_{j}[n]$ for some $j$ since $e_{i}[n]$ is currently an atom below $b_{\sigma}$. At each subsequent stage $n'>n$, one of the following may happen.
\begin{itemize}
    \item $l_{k}[n'-1]$ was split due to Step 3 of the construction for some $k\neq j$. Clearly, if $k>j$, then $l_{j}[n']=l_{j}[n'-1]$, as newly enumerated atoms have fresh indices. Thus, $e_{i}[n']=e_{i}[n'-1]=l_{j}[n'-1]=l_{j}[n']$. On the other hand, if $k<j$, then $l_{j}[n'-1]=l_{j-1}[n']$, as the atom that used to have the $k^{th}$ smallest index has been split into two new atoms with large indices. We then obtain that $e_{i}[n']=l_{j-1}[n']$. 

    \item There is some $k$ such that $l_{k}[n'-1]$ was split due to Step 4 of the construction. Recall that when splitting atoms in Step 4, the total number of atoms within $B$ do not change; each atom is split into a single new atom and a copy of $\interval(\eta)$. The temporary atoms in the latter are ignored as the construction will turn this into the atomless Boolean algebra in the limit. In addition, since the atoms are split according to their position in $L$, even though $l_{k}[n']\neq l_{k}[n'-1]$, the fresh atom below $l_{k}[n'-1]$ still possesses the $k^{th}$ smallest index in $B$. In particular, if $e_{i}[n'-1]=l_{j}[n'-1]$ and $l_{j}[n'-1]$ was split at stage $n'$ by Step 4 of the construction, $e_{i}[n']=l_{j}[n']$.
\end{itemize}
Thus, at each stage $n'>n$, $e_{i}[n']=l_{j}[n']$ or $e_{i}[n']=l_{j-1}[n']$. Therefore, there must exist $n^{*}$ and $j$ such that $e_{i}[n']=l_{j}[n']$ for all $n'\geq n^{*}$, and since Step 4 of the construction eventually stops splitting $l_{j}$ for each $j$, the indices of atoms below each $b_{\sigma}$ eventually stabilises.

Finally, we show that $B\cong\interval(\omega+\eta)$. For each $x\in\omega$, recall that there are only finitely many labels $\sigma$ of length $x+1$ such that $\hat{a}_{\sigma}$ is finite; these are exactly the labels with final entry $\langle s,i\rangle$ for some $s<s_{x}^{*}$ and $i=1$ or $2$. These finitely many labels are also exactly the labels for which $\hat{b}_{\sigma}$ is finite. This implies that for each $i\in\omega$, $c_{i-1}$ (let $c_{-1}$ be the top element of $B$) splits into $c_{i}$, a copy of $\interval(\eta)$ and some finite Boolean algebra. Furthermore, by Step 3 of the construction, we also obtain that the total number of atoms below $c_{i}^{*}$ is finite and larger than the number of atoms below $a_{\langle s_{0}^{*},1\rangle,\langle s_{1}^{*},1\rangle,\dots,\langle s_{i-1}^{*},1\rangle}\wedge\neg a_{\langle s_{0}^{*},1\rangle,\langle s_{1}^{*},1\rangle,\dots,\langle s_{i}^{*},1\rangle}$.

\

\emph{Computing $g^{*}(x)$:} Given an isomorphism $h:A\to B$ and $x\in\omega$, compute $h^{-1}(c_{i})$ for each $i\leq x$. At least one of these images should consist of a generator with label $\sigma$ having prefix $\langle s_{0}^{*},1\rangle^{\frown}\dots^{\frown}\langle s_{x}^{*},1\rangle$. We proceed via induction. If $h^{-1}(c_{0})$ does not contain $a_{\langle s_{0}^{*},1\rangle}$, then $h^{-1}(c_{0}^{*})$ must be contained in $\neg a_{\langle s_{0}^{*},1\rangle}$. However, by the construction, there are more atoms below $c_{0}^{*}$ than there are in $\neg\hat{a}_{\langle s_{0}^{*},1\rangle}$. Inductively suppose that $h^{-1}(c_{i})$ contains $a_{\langle s_{0}^{*},1\rangle,\dots,\langle s_{i}^{*},1\rangle}$. Since there are more atoms below $c_{i+1}^{*}$ than there are below $a_{\langle s_{0}^{*},1\rangle,\langle s_{1}^{*},1\rangle,\dots,\langle s_{i}^{*},1\rangle}\wedge\neg a_{\langle s_{0}^{*},1\rangle,\langle s_{1}^{*},1\rangle,\dots,\langle s_{i+1}^{*},1\rangle}$, it follows that $h^{-1}(c_{i+1})$ must contain $a_{\langle s_{0}^{*},1\rangle,\dots,\langle s_{i+1}^{*},1\rangle}$. Thus, one of $h^{-1}(c_{i})$ for $i\leq x$ contains a label with index $s\geq s_{x}^{*}$. For such an index $s$, by running the enumeration of $A$ up to stage $s$, we may then obtain the indices of at least $2^{s_{x}^{*}-1}$ many atoms. We note here that not all generators which appear to be atoms at stage $s$ may remain as atoms in the limit, however, it suffices that at least $2^{s_{x}^{*}-1}$ many of them will actually be atoms. We may further assume that $x$ is the largest for which $s_{x}^{*}\leq s$, otherwise, we would accordingly obtain the indices of at least $2^{s_{y}^{*}-1}$ many atoms in $A$ where $y$ is the largest such that $s_{y}^{*}\leq s$. The property that we require is that the number of (guaranteed) atoms is larger than $\langle x,s\rangle$. Given the indices of these atoms, compute $h$ on all of them. By pigeonhole principle, at least one of these should map to an atom within $B$ with index $t\geq t_{x,s}^{*}$ (the index of the $\langle x,s\rangle$-th smallest atom in $B$). Using these values of $s$ and $t$, we may then compute $g^{*}(x)$ by computing $g(x,s,t)$. Lemma~\ref{lem:intompeta} is proved.
\end{proof}

Theorem~\ref{thm:d3ba} is proved.

\section{Trees as Partial Orders}\label{sec:trees}

\begin{theorem}\label{thm:treepres}
    Every computable tree with at least one node having infinitely many successors has a punctual presentation.
\end{theorem}

\begin{proof}
Let $T$ be a computable tree. Non-uniformly fix a node $t^{*}$ with infinitely many successors. During the construction, we build both the punctual copy $A$ and maintain a computable isomorphism $f$ from $T$ to $A$. The main challenge in producing a punctual copy is that we need to be able to enumerate elements `quickly'. Thus, we must have some part of the structure dedicated to `padding', while waiting for the `slow' computable copy to reveal its structure. To this end, we use the node with infinitely many successors as our `padding' part. In $A$, let $a^{*}$ be a node with infinitely many successors. Then we define $f(t^{*})=a^{*}$. At each stage $s$, we enumerate a new node $a_{s}$ and declare $a^{*}<a_{s}$. The strategy now is to ensure that $f(T[s])$ is always a substructure of $A[s]$ and $A[s]\setminus(f(T[s]))$ consists only of nodes which are currently successors of $a^{*}$.

Suppose inductively that the above holds for all stages $<s$. If no new element is enumerated into $T$ at stage $s$, then simply enumerate $a_{s}$ into $A$ and proceed to the next stage. It is easy to see that if this happens, then the inductive hypothesis still holds. We may thus suppose that some element $t$ is enumerated into $T$ at stage $s$. By convention, we further assume that at most one such $t$ may enter $T$ at each stage. Consider the cases as follows.
\begin{itemize}
    \item If $t$ is currently a successor of $t^{*}$ and there does not exist any $t'\in T[s-1]$ for which $t^{*}<t<t'$, then define $f(t)=a_{i}$ for the least $i$ such that $a_{i}$ is not yet in the range of $f$. If such an element does not exist, simply enumerate a fresh $a_{i}$ and define it to be the $f$ image for $t$.

    \item If $t$ is currently a successor of $t^{*}$ and there exists some $t'\in T[s-1]$ for which $t^{*}<t<t'$, then enumerate a new element $a$ into $A$, defining $f(t)=a$. Furthermore, for any $a'$ currently in the range of $f$, the relation between $a,a'$ may be inherited directly from the relation between $t$ and $f^{-1}(a')$. For any other $a'\in A[s-1]\setminus(f(T[s-1]))$, define $a$ and $a'$ to be incomparable. This ensures that the relation on $A$ is primitive recursive.

    \item If $t$ is currently not a successor of $t^{*}$, then enumerate a new element $a$ into $A$, defining $f(t)=a$. As before, for any $a'\in f(T[s-1])$, the relation between $a,a'$ is inherited from the relation between $t,f^{-1}(a')$. For any $a'\in A[s-1]\setminus(f(T[s-1]))$, if $t<t^{*}$, then define $a<a'$. Otherwise, define $a$ to be incomparable with $a'$.
\end{itemize}
From the definition of $f$, it is clear that it is injective. It remains to argue that $f$ is surjective and order preserving. Let $a\in A$ be given. From the construction, a node $a$ is enumerated only for one of two reasons. First, $a=a_{i}$ for some $i$, and second, $a$ was enumerated to serve as the $f$ image for some $t$. In the latter case, it is obvious that $a$ is in the range of $f$. On the other hand, for $a=a_{i}$, by the stage at which $t^{*}$ has at least $i+1$ distinct successors, $a_{i}$ must have entered the range of $f$. Therefore, $f:T\to A$ is a bijection. It is also easy to see from the construction above that $f$ is order preserving, and thus, we obtain that $A\cong T$.
\end{proof}

We note that there exists a computable tree (viewed as a partial order) that does not have punctual copies (Theorem~8 in~\cite{BKW-ta}).

\begin{theorem}\label{thm:d1tree}
    Let $T$ be a punctual tree of finite height such that one of the following holds.
    \begin{itemize}
        \item $T$ has at least two nodes each with infinitely many successors.
        \item $T$ has exactly one node with infinitely many successors, each of which possessing its own successor.
    \end{itemize}
    Then $\prcat(T)\subseteq\cone(\Delta_{1}^{0})$.
\end{theorem}

\begin{proof}
Let $T$ be a computable tree that is not punctually categorical. In building punctual presentations $A,B\cong T$, we adopt the strategy as described in the proof of Theorem \ref{thm:treepres}. Recall that all that is necessary is for successors to be enumerated quickly for a node with infinitely many successors. In constructing $A,B$, we shall ensure that such a property holds. Let $g$ be a total computable function. Recall that it suffices to encode the stage at which $g(x)\downarrow$ into any isomorphism $h:A\to B$ to show that $\prcat(T)\subseteq\cone(\Delta_{1}^{0})$.

\

\emph{Constructing $A,B$ with (at least) two nodes having infinitely many successors:} Fix two nodes $a^{*}$ and $a^{**}$ in $A$, to be nodes with infinitely many successors. As described previously, at each stage of the construction, both $a^{*}$ and $a^{**}$ will receive fresh successors. Notice that because of the way we ensure $A\cong T$ (see Theorem \ref{thm:treepres}), the nodes enumerated as successors of $a^{*}$ or $a^{**}$ need not remain as successors. However, we at least have that these elements all remain incomparable. Furthermore, for any $x$, it is primitive recursive to obtain the indices of $x$ many incomparable descendants for both $a^{*}$ and $a^{**}$.

Within $B$, fix one node $b^{*}$ to be a node with infinitely many successors. To keep $B$ punctual, at each stage of the construction, enumerate a fresh successor for $b^{*}$. Obviously, the intention here is to use the descendants of all other nodes with infinitely many successors to encode the stages at which $g\downarrow$. More specifically, for every other node $b\neq b^{*}$, we ensure that $b$ has at most $x$ many incomparable descendants as long as $g(x)\uparrow$. Since $g$ is assumed to be total, this delay is only finite, and is easily seen to be compatible with the strategy to ensure $B\cong T$ (as in Theorem \ref{thm:treepres}). In particular, for any other node $b\neq b^{*}$ with infinitely many successors, the sub-tree above $b$ has width $x$ only after $g(y)\downarrow$ for all $y\leq x$.

Given any isomorphism $h:A\to B$, and $x\in\omega$, we may recover $g(x)$ with the following procedure. Compute $h$ on $x$ many incomparable descendants of $a^{*}$ and $a^{**}$ respectively. Since $h$ is an isomorphism, either $h(a^{*})\neq b^{*}$ or $h(a^{**})\neq b^{*}$. Since the subtree above $b\in B$ for any $b\neq b^{*}$ grows to width $x$ only after $g(y)\downarrow$ for all $y\leq x$, by pigeonhole principle, at least one of the $x$ many incomparable descendants of $a^{*}$ or $a^{**}$ must map to a node enumerated into $B$ after $g(x)\downarrow$. By convention, such a node in $B$ has index at least as large as the stage at which $g(x)\downarrow$, thus allowing us to recover $g(x)$.

\

\emph{Constructing $A,B$ with exactly one node having infinitely many successors, infinitely many of which possessing its own successor:} Fix the nodes $a^{*},b^{*}$ as the unique nodes with infinitely many successors in $A$ and $B$ respectively. At every even stage $s$ of the construction, enumerate nodes $a_{s,0},a_{s,1}$ such that $a^{*}\leq a_{s,0}\leq a_{s,1}$, and at every odd stage of the construction, enumerate a single new successor for $a^{*}$. The idea here is that we `quickly' produce the successors of $a^{*}$ with their own successors, while using the successors enumerated during odd stages to keep $A\cong T$. Since the assumption is that there are infinitely many successors of $a^{*}$ with their own successors, this enumeration of the sub-tree above $a^{*}$ is still compatible with the strategy in keeping $A\cong T$ as described in Theorem \ref{thm:treepres}.

In $B$, we enumerate a new successor for $b^{*}$ at each stage to keep $B$ punctual. The idea here is to delay growing the height of each sub-tree rooted at a successor of $b^{*}$ until $g\downarrow$. Let $b_{x}$ denote the $x^{th}$ node enumerated as a temporary successor of $b^{*}$. The strategy to ensure $B\cong T$ might dictate that we enumerate some new element $b$ such that either $b^{*}\leq b\leq b_{x}$ or $b_{x}\leq b$. Let such a $b$ be the very first element requested to be enumerated. We shall delay such an enumeration until $g(y)\downarrow$ for all $y\leq x$. In particular, for any $x\in\omega$, if $g(x)$ has yet to halt, then $b_{x}$ is temporarily a successor of $b^{*}$, and $b_{x}$ has no descendants. Just as before, since $g$ is total, all such delay for the sake of encoding $g$ is only finite and is thus consistent with the strategy in keeping $B\cong T$.

Given an isomorphism $h:A\to B$ and $x\in\omega$, compute $h(a_{2y,0})$ and $h(a_{2y,1})$ for each $y\leq x$. Since $h$ is an isomorphism, and $a^{*}\leq a_{2y,0}\leq a_{2y,1}$, it must be that $b^{*}\leq h(a_{2y,0})\leq h(a_{2y,1})$ (note that $h(a^{*})=b^{*}$). In addition, since $y\neq y'$ implies that $a_{2y,i}$ and $a_{2y',j}$ are incomparable for any $i,j\in\{0,1\}$, the image of $h(a_{2y,0})$ and $h(a_{2y,1})$ must map to $x+1$ many incomparable `branches' above $b^{*}$. By construction of $B$, at least one of $h(a_{2y,0})$ and $h(a_{2y,1})$ has to map to some element enumerated after $g(x)\downarrow$, as no nodes are ever enumerated between $b^{*}$ and $b_{z}$, or as descendants of $b_{z}$ for all $z\geq x$, until after $g(x)\downarrow$. Theorem~\ref{thm:d1tree} is proved.
\end{proof}

Applying the characterisation of computably categorical trees in \cite{lmms05,Miller-05} together with Theorem \ref{thm:d1tree} gives the following.

\begin{corollary}
    Let $T$ be a computably categorical tree that is not punctually categorical, then $\prcat(T)=\cone(\Delta_{1}^{0})$.
\end{corollary}

Since $\prcat(T)\subseteq\cone(\Delta_{1}^{0})$ implies that $T$ is not punctually categorical, we may also obtain the following.

\begin{corollary}
    Let $T$ be a punctual tree of finite height. $T$ is punctually categorical iff $T=S\sqcup F$ where $S$ is an infinite star, and $F$ is a finite set, or $T$ is finite.
\end{corollary}

\bibliographystyle{alpha}
\bibliography{mybib}

\end{document}